\documentclass[13pt,a4paper]{amsart}

\usepackage{amssymb, amsmath}

\usepackage{amscd}
\usepackage{hyperref}
\usepackage{todonotes}
\usepackage{enumerate}
\usepackage{appendix}

\let\phi\varphi

\newcommand{\ze}{\zeta}

\newcommand{\bdet}{\textbf{det}}

\usepackage{amssymb, amsfonts, amsmath, amsthm, bbm} 
\usepackage{colonequals} 
\usepackage{comment}
\usepackage{mathtools}
\usepackage[australian]{babel}
\usepackage{mathrsfs}
\usepackage{tensor}
\usepackage{tikz-cd}
\usepackage{rotating}
\usetikzlibrary{matrix,arrows}
\usepackage[margin=1.2in]{geometry}

\newtheorem{theorem}{Theorem}[section]

\newtheorem{proposition}[theorem]{Proposition}
\newtheorem{corollary}[theorem]{Corollary}
\newtheorem{thm-and-def}[theorem]{Theorem and Definition}

\theoremstyle{definition}

\newtheorem{remark}[theorem]{Remark}
\newtheorem{notation-and-rem}[theorem]{Notation and Remark}

\newcommand{\bp}{\boldsymbol{p}}

\newcommand{\bD}{\boldsymbol{D}}

\begin{document}
\title{Geometry of solutions to the c-projective metrizability equation}

\author{Keegan J. Flood and A. Rod Gover}

\address{K.J. F.:Department of Mathematics and Statistics\\
  Masaryk University\\
  Kotla\'{a}\v{r}sk\'{a} 2\\
  Brno 611-37\\
  Czech Republic}
  \email{flood@math.muni.cz}
\address{A.R.G.:Department of Mathematics\\
  The University of Auckland\\
  Private Bag 92019\\
  Auckland 1142\\
  New Zealand} 
\email{r.gover@auckland.ac.nz}

\begin{abstract} 
On an almost complex manifold, a quasi-K\"{a}hler metric, with canonical connection in the
c-projective class of a given minimal complex connection, is
equivalent to a non-degenerate solution of the c-projectively invariant
metrizability equation. For this overdetermined equation, replacing this maximal rank condition on solutions with a nondegeneracy condition on the prolonged system yields a strictly wider class of solutions with non-vanishing (generalized) scalar curvature. We study the geometries induced by this class of solutions. For each solution, the strict point-wise signature partitions the underlying manifold into strata, in a manner that generalizes the model, a certain Lie group orbit decomposition of $\mathbb{CP}^m$. We describe the smooth nature and geometric structure of each strata component, generalizing the geometries of the embedded orbits in the model. This includes a quasi-K\"{a}hler metric on the open strata components that becomes singular at the strata boundary. The closed strata inherit almost CR-structures and can be viewed as a c-projective infinity for the given quasi-K\"{a}hler metric.
\end{abstract}


\thanks{K. J. Flood gratefully acknowledges support from the Czech Science Foundation (GA\v{C}R) Grant 20-11473S and A. R. Gover gratefully acknowledges support from the Royal
  Society of New Zealand via Marsden Grants 16-UOA-051 and 19-UOA-008.}

\maketitle

\section{Introduction}\label{intro}

Given a smooth manifold $M$, a {\em projective structure} is an
equivalence class of torsion-free affine connections $\bp$ that have
the same geodesics as unparametrized curves. A {\em projective
  manifold} is a smooth manifold equipped with a projective class
$(M,\bp)$.  A natural question is whether such a structure is
metrizable, i.e., is there is a metric $g$ on $M$ whose Levi-Civita
connection $\nabla^{g}$ lies in the projective class $\bp$.  By
\cite{Mik,Sinj} this non-linear problem can instead be rephrased in
terms of solutions to the projectively invariant linear PDE
\begin{equation}\label{meteq}
\operatorname{trace-free} (\nabla_a \zeta^{bc}) = 0,
\end{equation}
where we employ the Penrose abstract index notation and $\nabla\in
\bp$.  The projective manifold is metrizable if and only if there is a
non-degenerate symmetric contravariant 2-tensor $\zeta^{bc}$
satisfying \eqref{meteq}, with inverse of the metric given by
$g^{bc}=\operatorname{sgn}(\tau) \tau \zeta^{bc}$, where $\tau$ is a
suitable determinant of $\ze$.  The study of this equation and related
topics has led to considerable recent progress \cite{BM, BDE, DE,
  EM, FM, GM, KM, Mat2, MR, Mettler}. There is growing interest in an
analogue in the setting of what is called c-projective geometry, see
e.g. \cite{CEMN, CG4, Mettler2}, and this is  what we take up here.

On an almost complex $2m$-manifold $(M,J)$, an \emph{almost c-projective
structure} is an equivalence class of affine connections $\bD$ which
preserve $J$, which have minimal torsion in the sense that the only
non-vanishing component of their torsion is the Nijenhuis tensor
$N^J$, and which have (up to reparametrization) the same $J$-planar
curves (a complex analogue of geodesics).  Here and throughout $m\geq
2$.  The analogue of metrizability, in the almost c-projective
setting, then, is whether there is exists a Hermitian metric on $M$
which is preserved by a connection in the c-projective class $\bD$.
Equivalently, an almost c-projective manifold is metrizable if there
exists a non-degenerate solution to the c-projectively invariant
linear PDE
\begin{equation}
\operatorname{trace-free} (\nabla_a \zeta^{bc}) = 0,
\end{equation}
where this is a complex trace and $\nabla\in \bD$. 
Explicitly, in real terms, this is given by
\begin{align}\label{cproj-meteq}
\nabla_{c} \zeta^{ab} - \frac{1}{m}\delta_{c}^{(a}\nabla_{d}\zeta^{b)d} - \frac{1}{m}J_{c}^{(b}J^{a)}_e\nabla_{d}\ze^{ed} = 0.
\end{align}
where $\zeta$ is a density weighted Hermitian form on $T^*M$.
Equation \eqref{cproj-meteq} is termed the {\em c-projective metrizability
  equation}.  The inverse metric on $M$ is then given by
$$g^{bc}=
\operatorname{sgn}(\operatorname{det}(\ze))
\operatorname{det}(\ze)\zeta^{bc},
$$
where, again, a suitable notion of determinant is involved.

Here we are interested in more general solutions to this c-projective metrizability equation \eqref{cproj-meteq}. In
particular, we obtain a result that extends, to generic solutions, a
result from \cite{CGH2} concerning the restricted class of so-called
{\em normal solutions}. Specifically our aim is to identify and
understand the smooth structure and geometry induced on the different
sets where a solution to \eqref{cproj-meteq} is non-degenerate and,
respectively, degenerate - we shall extend the terminology from
\cite{CGH2} and call this a {\em curved orbit decomposition}.  At
points where the solution is non-degenerate it induces a metric as
previously noted. But at points where a solution is degenerate there
is not, in general, a metric, since the metric becomes singular on
this set.  But, under suitable assumptions, the degeneracy locus of
the solution does inherit a rich geometric structure.  In particular,
it has a hypersurface type CR structure, for which the Levi-form
(arising as usual for an embedded hypersurface) can be seen to be
compatible with the metric defined away from the degeneracy locus of
the solution.  This work gives an alternative approach to the
c-projective compactification of complete non-compact
pseudo-Hermitian metrics, as developed and studied in \cite{CG4}.  The
problem we address is a special case of a more general phenomenon
which we describe below.

Natural overdetermined partial differential equations govern a huge
variety of geometric structures \cite{BMMR,CG2,CurGov,EM,GM,Jez,Pap}
on smooth manifolds.  It has long been known that features of
solutions to such equations can partition the manifold
\cite{BL,Derd,Leit2,Lis}, but only recently tools have been developed
for fully understanding the geometries on the more singular components
in a way that relates them to the ambient structures. In fact the components of
the partition can appear radically different to each other, but the link between them becomes clear when viewing them
via prolongations of the solution to the relevant geometric PDE, see e.g \cite{GNW,GPW}.
A reason that this is important is that one can exploit these
relationships to smoothly relate the distinct components of the
partition and thus study the geometry on one component by means of an
adjacent component, as seen in the geometric holography program
e.g.\ \cite{AGW,FG,GLW,GZ,Juh,MM}.

Hence, given a solution to an
overdetermined partial differential equation on a smooth manifold, the
key problem is to determine the basic data of the components of
the partition (e.g. are they smoothly embedded submanifolds of some
dimension or rather more complicated variety type structures?), then
to determine the geometric structures thereon. Finally, one wants to
usefully understand the relationship between the geometric structures on
neighboring components of the partition.

It turns out that for a broad class of natural overdetermined linear
partial differential equations, and then a class of solutions to these equations ,
one can obtain remarkably general results.  These are for what are
called {\em normal solutions} to {\em first BGG equations} . In
\cite{CGH1,CGH2} it is shown that the stratifications arising from
solutions to these must be locally diffeomorphic to stratifications
arising from group orbit decompositions of homogeneous model
geometries. Moreover, the components of the partition carry Cartan
geometries that are curved analogues of the homogeneous geometries on
the corresponding partition of components of the model.  Unfortunately,
the methods utilized in these sources applies only to solutions which
correspond to Cartan holonomy reductions.  Thus it is important to
establish to what extent similar results might be deduced, by
different methods, for more general solutions. This question is
treated for the equation \eqref{cprojmeteq} (which is an
example of a first BGG equation) in the present article, following to
an extent the ideas and the progress in \cite{FloodGov,GPW}.

A standard approach to studying and treating overdetermined equations
is via differential prolongation, see e.g.\ \cite{BCEG}. The
c-projective tractor calculus (cf. \cite{CEMN,CG4}) is a natural tool
for developing and organizing the prolonged system of the
c-projectively invariant equation (\ref{cprojmeteq}). One reason for
this is that, since it is a first BGG equation
\cite{CGH1,CGH2,CSS,CSouc}, the (first) BGG splitting operator (a
canonical invariant differential operator) maps, loosely speaking, a
potential solution of \eqref{cprojmeteq} to its prolonged variable system. We
denote this c-projectively invariant second order operator
$\zeta\mapsto L(\zeta)$, where $L(\zeta)$ takes values in the bundle
$\mathcal{H}^*$ of Hermitian forms on the standard c-projective
cotractor bundle. If $L(\zeta)$ is parallel for the tractor
connection then  the solution $\zeta$ is said to be {\em normal}, but
we consider a more general class of solutions here. This BGG machinery
is introduced in Section \ref{cprojbgg} below.

There is a canonical
(c-projectively invariant) real-valued determinant on sections of $\mathcal{H}^*$ so we
consider the composite map
\begin{equation}\label{det}
\zeta\mapsto L(\zeta)\mapsto \det (L(\zeta)),
\end{equation}
which takes a solution $\zeta$ of \eqref{cprojmeteq} to the determinant of its prolonged system.
If $\zeta$ is a non-degenerate solution of
(\ref{cprojmeteq}) then, up a non-zero constant, $\det (L(\zeta))$ is the scalar curvature of the metric $g$
with inverse $g^{-1}=\operatorname{sgn}(\tau) \tau \zeta$
\cite{CG4}, where $\tau$ a suitable determinant of $\zeta$. 
But (\ref{det}) is well-defined even
where $\zeta$ is degenerate. 
Thus it is natural to consider
solutions $\zeta$ of equation (\ref{meteq}) satisfying the
condition that $\det (L(\zeta))$ is nowhere zero, i.e. with $L(\zeta)$
non-degenerate, but with no {\em a priori} restriction on the rank of
$\zeta$. 
This generic condition is a
generalization of constant scalar curvature, where $\zeta$ can have a non-empty degeneracy locus. 
Such considerations lead to the following result.

\begin{theorem}\label{one}
Let $(M,J,\bD)$ be a almost c-projective manifold with real dimension $2m$ equipped with a solution $\ze \in \Gamma (\operatorname{Herm}(T^*M)\otimes \mathcal{E}(-1,-1)_{\mathbb{R}})$ of the metrizability equation such that $L(\ze)\in\Gamma(\mathcal{H}^*)$ is non-degenerate as a pseudo-Hermitian form on the cotractor bundle. If $L(\ze)$ is definite then the degeneracy locus $\mathcal{D}(\ze)$ is empty and $(M, J, \bD,\zeta)$ is
  a quasi-K\"{a}hler manifold with inverse Hermitian metric $g^{-1}=\operatorname{sgn}(\tau)\tau
  \ze$ where $\tau={\normalfont \bdet}(\ze)$. If $L(\ze)$ has signature $(p+1,q+1)$, with $p,q\geq 0$, then $\mathcal{D}(\ze)$ is either empty or it is a smoothly embedded separating real hypersurface such that the following hold: 
 \begin{enumerate}[(i)]
\item $M$ is stratified by the strict signature of $\ze$ as a (density
  weighted) Hermitian form on $T^*M$ with curved orbit decomposition given by
$$
M=  \coprod\limits_{i \in \{+,0,-\}} M_i
$$
%%where $\tau >0$ on $M_+$  $\tau <0$ on $M_-$, and $\tau = 0$ on $M_0$. \\
 where $\zeta$ has signature $(p+1,q)$,
$(p,q+1)$,and $(p,q,1)$ on $TM$ restricted to $M_{+}$, $M_{-}$, and $M_0$, respectively. 
\item On $M_\pm$, $\zeta$ induces a quasi-K\"{a}hler metric $g_{\pm}$ with nonvanishing scalar curvature $R^{g_{\pm}}$, with the same signature as $\zeta$, with inverse $g^{-1}_{\pm} = \operatorname{sgn}(\tau)\tau \zeta|_{M_{\pm}}$ where $\tau={\normalfont \bdet}(\ze)$. 
\item If $M$ is closed, then the components $(M \backslash M_{\mp},J,\bD)$ are 
c-projective compactifications of $(M_{\pm}, J, \nabla^{\ze})$, with boundary $M_0$.
\item $M_0$ inherits a signature $(p,q)$ almost CR structure of hypersurface type.
\end{enumerate}
\end{theorem}

A smoothly embedded submanifold of real codimension 1 will be referred to as a hypersurface.
Note that each of the components $M_+$,
$M_0$, and $ M_{-}$ in the above theorem need not be
connected. 
We denote the
signature of a real symmetric bilinear form by $(p,q,r)$, where $p,q$
and $r$ are the number, counting multiplicity, of positive, negative,
and zero eigenvalues, respectively, of any matrix representing the
form once a basis has been chosen. 
When $r=0$ we omit it.

The Fubini-Study metric is a compact homogeneous
model for Hermitian geometry. 
There are corresponding compact models
for the geometries discussed in Theorem \ref{one} demonstrating that c-projective manifolds equipped with solutions of
\eqref{cprojmeteq} satisfying the given constant rank conditions on their
prolonged systems $L(\zeta)$ exist and are of interest. 
The models for the structures in Theorem \ref{one} are treated in Section \ref{model3} and from them we glean deeper insight into the result. 

Further motivation comes from \cite{CG4}, wherein it is shown that, given a manifold with boundary whose interior is equipped with a pseudo-Hermitian metric satisfying a
non-vanishing scalar curvature condition and whose its c-projective
structure extends to the boundary but whose canonical
connection does not extend to any neighborhood of the boundary, then the metric is c-projectively compact of order 2. 
Examples of c-projectively
compactified metrics discussed in \cite{CG4}
demonstrates the existence of curved examples of the structures considered in Theorem \ref{one}.

The non-degeneracy assumption on $L(\zeta)$, in Theorem \ref{one}, is
a constant $G$-type assumption, where we have used the terminology of
\cite{CGH1,CGH2}. Constancy of $G$-type holds for normal solutions on
connected manifolds, but it is not known to hold for general
solutions. As discussed in \cite{FloodGov} (using results from
\cite{KW1,KW2}), the fixed $G$-type assumption is necessary to get a
coherent theory, as the zero locus of the scalar curvature can be very
poorly behaved. In particular, it need not be a submanifold.

The structure of the article is as follows. 
In the Section
\ref{background} we briefly review c-projective tractor calculus,
c-projective compactification, and BGG machinery. 
These provide the framework and
computational tools we utilize. 
In Section \ref{cprojresults} we state and
prove the main results.

\section{C-projective geometry}\label{background}
\setcounter{section}{2}

In this section, we describe the necessary background from c-projective geometry. We draw from the main monograph on the subject \cite{CEMN} as well as from \cite{CG4} since we will need both the (predominantly) complex viewpoint of the former as well as the real viewpoint of the latter. Let $(M,J)$ be an almost complex manifold of dimension $n=2m\geq 4$.

The complexified tangent bundle $\mathbb{C}TM$ and complexified cotangent bundle $\wedge^1M$ decompose into the following direct sums
\begin{align}\label{directsumdecomp}
\mathbb{C}\mathcal{E}^a\colonequals & \mathbb{C}TM = T^{1,0}M \oplus T^{0,1}M \\
\mathbb{C}\mathcal{E}_a\colonequals & \wedge^1M = \wedge^{0,1}M \oplus \wedge^{1,0}M
\end{align}
where 
\begin{align*}
\mathcal{E}^{\alpha}\colonequals & T^{1,0}M  = \{ X\in \Gamma(TM): JX=iX \} \\
\mathcal{E}^{\overline{\alpha}}\colonequals & T^{0,1}M =  \{ X\in \Gamma(TM): JX=-iX \} \\
\mathcal{E}_{\overline{\alpha}} \colonequals & \wedge^{0,1}M = \{\alpha \in \Gamma(T^*M): J\alpha=-i\alpha \} \\
\mathcal{E}_{\alpha}\colonequals & \wedge^{1,0} M = \{\alpha\in \Gamma(T^*M) : J\alpha=i\alpha \}
\end{align*}
are the vector fields of type $(1,0)$ and $(0,1)$ and $1$-forms of type $(0,1)$ and $(1,0)$, respectively. There are conjugate linear isomorphisms $T^{1,0}M= \overline{T^{0,1}M}$ and $\wedge^{0,1}M=\overline{\wedge^{1,0}}M$. Observe that there are canonical pairings of $\mathcal{E}^{\alpha}$ and $\mathcal{E}^{\overline{\alpha}}$ with their respective duals $\mathcal{E}_{\alpha}$ and $\mathcal{E}_{\overline{\alpha}} $, which is compatible with the canonical complex pairing of $\mathbb{C}\mathcal{E}^a$ with $\mathbb{C}\mathcal{E}_a$. Note that we will be using lower case latin indices for real and complex vector fields and $1$-forms. 

We have the following complex linear projection maps: 
\begin{align*}
\mathbb{C}TM & \twoheadrightarrow T^{1,0}M  \ \ \ \ \ \ \ \ \ \ \ &&X^a  \mapsto \Pi_a^{\alpha}X^a \colonequals \frac{1}{2}(X-iJX) \\
 \mathbb{C}TM & \twoheadrightarrow T^{0,1}M  \ \ \ \ \ \ \ \ \ \ \  &&X^a  \mapsto \overline{\Pi}_a^{\overline{\alpha}}X^a \colonequals \frac{1}{2}(X+iJX),
\end{align*}
and their duals
\begin{align*}
\wedge^{1,0}M&\hookrightarrow \wedge^1M \ \ \ \ \ \ \ \ \ \ \ &&\omega_{\alpha}  \mapsto \Pi_a^{\alpha}\omega_{\alpha} \\
\wedge^{0,1}M&\hookrightarrow \wedge^1M \ \ \ \ \ \ \ \ \ \ \ &&\omega_{\overline{\alpha}}  \mapsto \overline{\Pi}_a^{\overline{\alpha}}\omega_{\overline{\alpha}}.
\end{align*}
Similarly, we have the inclusions:
\begin{align*}
T^{1,0}M & \hookrightarrow \mathbb{C}TM  \ \ \ \ \ \ \ \ \ \ \ &&X^{\alpha}  \mapsto \Pi^a_{\alpha}X^{\alpha} \\
T^{0,1}M & \hookrightarrow \mathbb{C}TM  \ \ \ \ \ \ \ \ \ \ \  &&X^{\overline{\alpha}}  \mapsto \overline{\Pi}^a_{\overline{\alpha}}X^{\overline{\alpha}} ,
\end{align*}
and their duals
\begin{align*}
\wedge^{1}M&\twoheadrightarrow \wedge^{1,0}M \ \ \ \ \ \ \ \ \ \ \ &&\omega_a  \mapsto \Pi^a_{\alpha}\omega_{\alpha} \\
\wedge^{1}M&\twoheadrightarrow \wedge^{0,1}M \ \ \ \ \ \ \ \ \ \ \ &&\omega_a \mapsto \overline{\Pi}^a_{\overline{\alpha}}\omega_{\overline{\alpha}}.
\end{align*}
These lead to the following identities
\begin{align*}
&\Pi_{\alpha}^a\Pi^{\beta}_a  = \delta^{\beta}_{\alpha} \ \ \ \ \ \ \ \ \ \ &&\Pi^{\alpha}_a\Pi^b_{\alpha}  = \frac{1}{2}(\delta^b_a-iJ^b_a) \ \ \ \ &&&\overline{\Pi}^{\overline{\alpha}}_a\overline{\Pi}^b_{\overline{\alpha}} = \frac{1}{2}(\delta^b_a+iJ^b_a)
\end{align*}
\begin{align*}
&\Pi_{\alpha}^aJ_a^b =i\Pi^b_{\alpha} \ \ \ \ \  &&\overline{\Pi}^a_{\overline{\alpha}}J^b_a = -i\overline{\Pi}^b_{\overline{\alpha}} &&& J^b_a\Pi^{\alpha}_b=i\Pi_a^{\alpha} &&&& J^b_a\overline{\Pi}^{\overline{\alpha}}_b=-i\overline{\Pi}_a^{\overline{\alpha}}.
\end{align*}
Complex valued differential forms can be naturally decomposed according to type e.g. 
\[
\wedge^2M = \wedge^{0,2}M\oplus\wedge^{1,1}M\oplus\wedge^{0,2}M.
\]

Although such characterizations quickly grow cumbersome for higher forms, $2$-forms are characterized as follows:
\begin{align*}
J_a^c\omega_{bc}=-i\omega_{ab} \Longleftrightarrow \omega_{ab} \ \operatorname{is\ type\ } (2,0)  \\
J_{[a}^c\omega_{b]c}=0 \Longleftrightarrow \omega_{ab} \ \operatorname{is\ type\ } (1,1)  \\
J_a^c\omega_{bc}=i\omega_{ab} \Longleftrightarrow \omega_{ab} \ \operatorname{is\ type\ } (0,2) 
\end{align*}
This characterization extends to appropriate almost complex vector bundle valued real forms. E.g., the Nijenhuis tensor $N^J\in \Omega^2(M,TM)$, which satisfies $N^J(X,JY)=JN^J(Y,X)$, is type $(0,2)$.

\subsection{Complex Connections}
Affine connections preserving $J$, i.e. satisfying $\nabla_{\mu}J\eta = J\nabla_{\mu}\eta$ for all $\mu,\eta\in \mathfrak{X}(M)$, are termed {\em complex connections}. It follows that an affine connection is complex if and only if its extension to a linear connection on $\mathbb{C}\mathcal{E}^a$ preserves types. The torsion $T_{ab}^c$ of a complex connection $\nabla$ naturally splits into types, with the $(0,2)$ component being precisely $-\frac{1}{4}N^J$. So a complex connection cannot be torsion-free unless its Nijenhuis tensor vanishes identically i.e. the almost complex structure is integrable. Given an almost complex manifold, there always exists a complex connection on it with torsion of type $(0,2)$ by \cite{Lichnerowicz}. The $(2,0)$ and $(1,1)$ components of the torsion can be removed via a suitable modification to the complex connection, but as an almost complex invariant the $(0,2)$ component may not be eliminated. 

The pseudo-Riemannian metrics of interest on almost complex manifolds are those which are Hermitian for J, i.e. satisfying $g_{ab}J^a_cJ^b_d=g_{cd}$. {\em Pseudo-K\"{a}hler metrics} are precisely the Hermitain metrics whose Levi-Civita connections are complex. Projective equivalence of two such pseudo-K\"{a}hler metrics on $(M,J)$ implies that they are in fact affinely equivalent \cite{Bochner}. Thus we must introduce a broader class of curves, the so-called {\em $J$-planar} curves. A $J$-planar curve is a curve $c:I\rightarrow M$ satisfying 
\[
\nabla_{\dot{c}}\dot{c}=\alpha \dot{c} + \beta J\dot{c}
\] 
for some $\alpha,\beta:I\rightarrow \mathbb{R}$. These are also commonly termed holomorphically flat curves \cite{OtsukiTashiro} or $h$-planar curves \cite{Mat5}. Clearly all curves are J-planar on almost complex manifolds of real dimension $2$. 

Consider $(\mathbb{CP}^n, J_{Can}, g^{FS})$, where $J_{Can}$ denotes the canonical complex structure and $g^{FS}$ denotes the Fubini-Study metric. Observing that the embedding of any complex line $\mathbb{CP}^1\hookrightarrow \mathbb{CP}^n$ is totally geodesic with respect to $\nabla^{FS}$, it follows (for details see Example 1 of \cite{Mat5}) that the $J$-planar curves on $(\mathbb{CP}^n, J_{Can}, [\nabla^{FS}])$ are precisely the curves in these linearly embedded copies of $\mathbb{CP}^1$.

We say that two complex connections $\nabla$ and $\tilde{\nabla}$ on an almost complex manifold $(M,J)$ are {c-projectively equivalent} if they have the same $J$-planar curves and the same torsion. Two such complex connections are explicitly related by
\begin{align*}\label{cprojchange}
& \tilde{\nabla}_{a}\eta^b = \nabla_{a}\eta^b + \Upsilon_c\eta^b - \Upsilon_cJ^c_aJ^b_d\eta^d + \Upsilon_c\eta^c\delta_a^b-\Upsilon_cJ^c_d\eta^dJ^b_a \\
&\tilde{\nabla}_{\alpha}\eta^{\gamma} =\nabla_{\alpha}\eta^{\gamma} + 2\Upsilon_{\alpha}\eta^{\gamma} + 2\delta^{\gamma}_{\alpha}\Upsilon_{\beta}\eta^{\beta} \\
%& \tilde{\nabla}_{\overline{\alpha}}\eta^{\overline{\gamma}} =\nabla_{\overline{\alpha}}\eta^{\overline{\gamma}} + 2\Upsilon_{\overline{\alpha}}\eta^{\overline{\gamma}} + 2\delta^{\overline{\gamma}}_{\overline{\alpha}}\Upsilon_{\overline{\beta}}\eta^{\overline{\beta}} \\
& \tilde{\nabla}_{\overline{\alpha}}\eta^{\gamma} =\nabla_{\overline{\alpha}}\eta^{\gamma}\\
%& \tilde{\nabla}_{\alpha}\eta^{\overline{\gamma}} =\nabla_{\alpha}\eta^{\overline{\gamma}} \\
& \tilde{\nabla}_{a}\nu_c = \nabla_{a}\nu_c - \Upsilon_a\nu_c + \Upsilon_bJ^b_aJ^d_c\nu_d - \Upsilon_c\nu_a+\Upsilon_dJ^b_a\nu_bJ^d_c \\
& \tilde{\nabla}_{\alpha}\nu_{\gamma} =\nabla_{\alpha}\nu_{\gamma} - 2\Upsilon_{\alpha}\nu_{\gamma}-2\nu_{\alpha}\Upsilon_{\gamma} \\
%& \tilde{\nabla}_{\overline{\alpha}}\nu_{\overline{\gamma}}  = \nabla_{\overline{\alpha}}\nu_{\overline{\gamma}} -2\Upsilon_{\overline{\alpha}}\nu_{\overline{\gamma}} -2\nu_{\overline{\alpha}}\Upsilon_{\overline{\gamma}} \\
& \tilde{\nabla}_{\overline{\alpha}}\nu_{\gamma}  = \nabla_{\overline{\alpha}}\nu_{\gamma}
%& \tilde{\nabla}_{\alpha}\nu_{\overline{\gamma}}  = \nabla_{\alpha}\nu_{\overline{\gamma}}
\end{align*}
for some one form $\Upsilon\in \Omega^1(M)$, where $\Upsilon_{\alpha}\colonequals \Pi_{\alpha}^a\Upsilon_a$ and $\Upsilon_{\overline{\alpha}}\colonequals \Pi_{\overline{\alpha}}^a\Upsilon_a$. We write $\tilde{\nabla}=\nabla + \Upsilon$ as a brief notation to indicate connections related as in the above formulae. Note that we follow the convention of \cite{CG4} in \eqref{cprojchange} rather than that of \cite{CEMN}.

In fact, we will only consider complex connections $\nabla$ with {\em minimal} torsion $T^{\nabla}=-\frac{1}{4}N^J$, we term these {\em minimal complex connections}. We write $\bD$ for an equivalence class of c-projectively related minimal complex affine connections and we call it an {\em almost c-projective structure}. We call a triple $(M,J,\bD)$ an {\em almost c-projective manifold}. If $J$ is integrable we call $(M,J,\bD)$ a {\em c-projective manifold}.

\subsection{C-projective densities}\label{cprojdensities}
We write $\mathcal{E}(m+1,0)\colonequals \Lambda^m_\mathbb{C}TM$ for the top complex exterior power of the tangent bundle. We will assume the existence of $(m+1)^{\operatorname{th}}$ roots of this bundle. In particular, this holds on the model $(\mathbb{CP}^n, J_{Can}, [\nabla^{FS}])$ and hence locally for all almost c-projective manifolds.  Assuming a choice $\mathcal{E}(1,0)$ of $(m+1)^{\operatorname{th}}$ root of $\mathcal{E}(m+1,0)$, denote the dual, conjugate, and dual conjugate to $\mathcal{E}(1,0)$ by $\mathcal{E}(-1,0)$, $\mathcal{E}(0,1)$, and $\mathcal{E}(0,-1)$, respectively. Forming tensor powers of these bundles gives complex density bundles $\mathcal{E}(k,l)$ of weight $(k,l)$ where $k,l\in \mathbb{Z}$. 

There is a natural inclusion $\mathcal{E}(-2m-2)\hookrightarrow \mathcal{E}(-m-1,-m-1)$ of the real densities of weight $(-2m-2)$ into the complex densities of weight $(-m-1,-m-1)$ as the real subbundle fixed by conjugation. The orientation on $(M,J)$, induced by the almost complex structure $J$, induces an orientation on the trivial bundle $\mathcal{E}(-2m-2)$ and so allows us to take arbitrary real roots of $\mathcal{E}(-2m-2)$ which give the usual real densities $\mathcal{E}(w)$ of weight $w\in \mathbb{R}$. Thus, for $w,w' \in \mathbb{R}$ such that $w-w'\in\mathbb{Z}$, we can define complex density bundle $\mathcal{E}(w,w')$. We denote the image of $\mathcal{E}(2w)$ under this inclusion map by $\mathcal{E}(w,w)_{\mathbb{R}}$.   

%Then for $s\in \Lambda_{\mathbb{C}}^mTM$ and $\Sigma\in\Lambda^{2m}TM$ c-projectively related connections transform via
%\begin{align*}
%\tilde{\nabla}_{a}\Sigma& =\nabla_a\Sigma + 2(m+1)\Upsilon_a\Sigma \\
%\tilde{\nabla}_{a}s & = \nabla_as + (m+1)\Upsilon_as-(m+1)i\Upsilon_bJ_a^bs.
%\end{align*}
%Similarly, for $\sigma\in \Gamma(\Lambda^mT^{1,0}M)$ and $\rho\in \Gamma(\Lambda^mT^{0,1}M)$, 
%\begin{align*}
%\tilde{\nabla}_{\alpha}\sigma& =\nabla_{\alpha}\sigma + 2(m+1)\Upsilon_{\alpha}\sigma \\
%\tilde{\nabla}_{\overline{\alpha}}\rho& =\nabla_{\overline{\alpha}}\rho + 2(m+1)\Upsilon_{\overline{\alpha}}\rho \\
%\tilde{\nabla}_{\overline{\alpha}}\sigma & = \nabla_{\overline{\alpha}}\sigma\\
%\tilde{\nabla}_{\alpha} \rho & = \nabla_{\alpha}\rho.
%\end{align*}
%These imply that for $\upsilon\in\Gamma(\mathcal{E}(w,w'))$ with $w,w'\in\mathbb{R}$ and $w-w'\in \mathbb{Z}$ we obtain
It is shown in \cite{CG4} that for $\upsilon\in\Gamma(\mathcal{E}(w,w'))$ with $w,w'\in\mathbb{R}$ and $w-w'\in \mathbb{Z}$
\begin{align}
\tilde{\nabla}_a\upsilon = \nabla_a\upsilon +(w+w')\Upsilon_a\upsilon-(w-w')i\Upsilon_bJ^b_a\upsilon.
\end{align}
In particular, for $\tau\in\Gamma(\mathcal(w,w)_{\mathbb{R}})$ with $w\in\mathbb{R}$ this reduces to 
\begin{align}
\tilde{\nabla}_a\tau = \nabla_a\tau +2w\Upsilon_a\tau.
\end{align}

\subsection{C-projective compactness}\label{cprojcompactness}
A {\em local
  defining function} for a hypersurface $\Sigma$ is a smooth function
$r:U \rightarrow \mathbb{R}$, defined on an open  subset $U$ of $M$,
satisfying $\mathcal{Z}(r) = \Sigma \cap U$ and $\mathcal{Z}(dr) \cap
\Sigma = \varnothing$ on $\Sigma \cap U$, where $\mathcal{Z}(-)$
denotes the zero locus. Then, extending this concept, a {\em local defining
  density of weight w} is a local section $\sigma$ of $\mathcal{E}(w)$
such that $\sigma = r\hat{\sigma}$, where $r$ is a defining function
for $\Sigma$ and $\hat{\sigma}$ is a section of $\mathcal{E}(w)$ that
is nonvanishing on $U$. 

Consider a smooth manifold with boundary, $\overline{M}=M\cup \partial M$, such that the interior $M$ is equipped with an almost complex structure $J$ and a minimal complex affine connection $\nabla$ on $TM$. The (minimal) complex connection $\nabla$ on $TM$ is said to be {\em c-projectively compact} of order $\alpha \in \mathbb{R}_+$ if and only if for any $x \in \partial M$ there is a neighborhood $U$ of $x$ in $\overline{M}$ and a defining function $\rho:U\rightarrow \mathbb{R}_{\geq 0}$ for $U\cap \partial M$ such that the c-projectively equivalent connection $\overline{\nabla}=\nabla + \frac{d\rho}{\alpha\rho}$ on $U\cap M$ smoothly extends to all of $U$, i.e $\overline{\nabla}_{\mu}\eta$ is smooth up to the boundary for all $\mu,\eta\in\mathfrak{X}(\overline{U})$. In what follows we will only be concerned with the case where $\alpha=2$, so we will often omit the order of the c-projective compactification. As in the case of projective compactification  (cf. \cite{CG3}) this notion is independent of choice of defining function. A connection $\nabla$ is said to be {\em special} if and only if there is a section $\tau\in\Gamma(w,w)_{\mathbb{R}}$ with $w\neq0$ such that $\tau$ is parallel for $\nabla$. This leads to the following, which is Proposition 6 of \cite{CG4}.
\begin{proposition}\label{cprojcompactnessprop}
Let $\overline{M}=M\cap\partial M$ be a smooth manifold of dimension $n=2m$ equipped with a special affine connection $\nabla$ on $TM$. Then the following hold:
\begin{enumerate}
 \item If $\nabla$ is c-projectively compact of order $2$ then a non-vanishing section of $\mathcal{E}(1,1)_{\mathbb{R}}$ which is parallel for $\nabla$ extends by zero to a defining density for $\partial M$.
 \item If the almost c-projective structure on $M$ determined by $\nabla$ smoothly extends to $\overline{M}$ and there exists a defining density $\tau\in\Gamma(\mathcal{E}(1,1)_{\mathbb{R}})$ for $\partial M$ that is parallel on $M$ for $\nabla$, then $\nabla$ is c-projectively compact of order 2.
 \end{enumerate}
\end{proposition}
Let $(M,\nabla)$ be a smooth manifold equipped with a complex connection. If there exists a smooth manifold with boundary $\overline{M}$ such that $\overline{M} =M\cup\partial M$ for which $\nabla$ is c-projectively compact we will say that $(\overline{M},[\nabla])$ is a {\em c-projective compactification} of $(M,\nabla)$.
 
\subsection{Admissible metrics}\label{admissiblemetrics}
How do metrics fit into the picture? We discussed earlier that the relevant pseudo-Riemannian metrics in almost c-projective geometry are those which are Hermitian with respect to $J$. Minimizing the torsion we come to the class of connections $\bD$, which in general have torsion of type $(0,2)$, and so cannot be the Levi-Civita connection. Fortunately, a minimal complex connection preserving a pseudo-Riemannian metric, that is Hermitian for $J$, is uniquely determined. Such a connection need not exist in general. 

On an almost complex manifold $(M,J)$ a pseudo-Riemannian metric $g$ that is Hermitian for $J$ is said to be {\em admissible} if and only if there is a minimal complex affine connection on $(M,J)$ preserving $g$. If such a connection exists it is termed the {\em canonical connection} associated to $g$. By Proposition 4.1 of \cite{CEMN} or Proposition 7 of \cite{CG4} a pseudo-Riemannian metric on an almost complex manifold $(M,J)$ that is Hermitian for $J$ is admissible if and only if it is quasi-K\"{a}hler in the sense of Gray-Hervella \cite{GrayHervella}. 

In the notation of Gray-Hervella, quasi-K\"{a}hler is $\mathcal{W}_1\oplus\mathcal{W}_2$. $\mathcal{W}_1$ denotes the class of nearly K\"{a}hler manifolds i.e. $d\omega=3\nabla\omega$ where $\omega$ is the {\em K\"{a}hler form} $\omega_{ab}=J_a^cg_{cb}$ and $\nabla$ is the canonical connection associated to $g$. $\mathcal{W}_2$ denotes the class of almost K\"{a}hler manifolds i.e. $d\omega=0$. If $J$ is integrable, $g$ is admissible if and only if it is pseudo-K\"{a}hler, i.e. $\nabla\omega=0$. 

\subsection{The c-projective Schouten tensor}
The curvature tensor, $R\indices{_{ab}^c_d}$, of a complex affine connection, $\nabla$, on an almost complex manifold $(M^{2m},J)$ satisfies $R\indices{_{ab}^c_i}J_d^i=R\indices{_{ab}^i_d}J^c_i$. Denoting the Ricci tensor by $R_{ab}=R\indices{_{ia}^i_b}$ we define the {\em Rho tensor} or {\em c-projective Schouten tensor} by 
\begin{align}
P_{ab}\colonequals \frac{1}{2(m+1)}(R_{ab}+\frac{1}{m-1}(R_{(ab)}-J^i_{(a}J^j_{b)}R_{ij})).
\end{align} 
Given a complex connection $\nabla$ with Schouten $P$, The Schouten $\overline{P}$ of the c-projectively related connection $\tilde{\nabla}=\nabla + \Upsilon$ is given by
\begin{align}\label{cprojschoutentransform}
\tilde{P}_{ab}=P_{ab}-\nabla_a\Upsilon_b+\Upsilon_a\Upsilon_b - J_a^iJ_b^j\Upsilon_i\Upsilon_j.
\end{align}
Writing $W\indices{_{ab}^c_d}$ for the Weyl curvature we have the following
\begin{align}\label{cprojschoutencurvature}
R\indices{_{ab}^c_d}=W\indices{_{ab}^c_d} +2\delta^c_{[a}P_{b]d}-2P_{[ab]}\delta_d^c -2J^i_{[a}P_{b]i}J^c_d-2J^c_{[a}P_{b]i}J^i_d.
\end{align}
%Noting that our convention for the Schouten is $\frac{1}{2}$ that in \cite{CEMN}. Both \eqref{cprojschoutentransform} and \eqref{cprojschoutencurvature} follow from Section 2.4 of \cite{CEMN}.
Observe that if the Ricci is Hermitian then $P_{ab}=\frac{1}{2(m+1)}R_{ab}$ and if the Ricci is symmetric then the Schouten is symmetric as well. 

\subsection{C-projective tractor bundle}
An almost c-projective manifold $(M,J,\bD)$ equipped with a choice of density bundle $\mathcal{E}(1,0)$ is equivalent to a Cartan geometry $(\mathcal{P}\twoheadrightarrow M,\omega)$ of type $(G, P)$ where $G = SL(m+1,\mathbb{C}) \cong SL(2m+2,\mathbb{J}) $, which we identify with its standard representation on $\mathbb{C}^{m+1}$, and $P\subseteq G$ is the isotropy group of a complex line through the origin in $\mathbb{C}^{m+1}$. Restricting this representation to $P$, call the restricted representation space $\mathbb{V}$. The corresponding tractor bundle is the {\em standard c-projective tractor bundle} $\mathcal{T}$, i.e.
\begin{align}
\mathcal{E}^{\mathscr{A}}=\mathcal{T}\colonequals \mathcal{P}\times_{P}\mathbb{V}.
\end{align}
Its dual, the {\em standard cotractor bundle}, is given by
\begin{align}
\mathcal{E}_{\mathscr{A}}=\mathcal{T}^*\colonequals \mathcal{P}\times_{P}\mathbb{V}^*.
\end{align}
We define the {\em standard complex tractor bundle} to be the $(1,0)$ component of the complexification of the real standard tractor bundle
\begin{align}
\mathcal{E}^A=\mathcal{T}^{1,0}\subset \mathbb{C}\mathcal{T},
\end{align}
We denote it's conjugate, dual, and dual conjugate by $\mathcal{E}^{\overline{A}}=\mathcal{T}^{0,1}=\overline{\mathcal{T}^{1,0}}$,  $\mathcal{E}_{A}=(\mathcal{T}^{1,0})^*$, and $\mathcal{E}_{\overline{A}}=(\mathcal{T}^{0,1})^*$, respectively. 

Recalling the various natural maps denoted by $\Pi$ from the beginning of Section \ref{background}, observe that there are, \emph{mutatis mutandis}, analogous natural complex linear inclusions and projections at the tractor level which satisfy similar identities. For instance, 
\begin{align*}
(\mathcal{T}^*)^{(1,0)}&\hookrightarrow \mathbb{C}\mathcal{T}^* \ \ \ \ \ \ \ \ \ \ \ &&u_{A}  \mapsto \Pi_{\mathscr{A}}^{A}u_{A} \\
(\mathcal{T}^*)^{(0,1)}&\hookrightarrow \mathbb{C}\mathcal{T}^* \ \ \ \ \ \ \ \ \ \ \ &&u_{\overline{A}}  \mapsto \overline{\Pi}_{\mathscr{A}}^{\overline{A}}u_{\overline{A}}.
\end{align*}
Note that we use capital script indices for the real standard (co)tractor bundle, and its complexification when no confusion can arise. The structure of the tractor bundles defined above can be described by the following short exact sequences
\begin{align*}
& 0\rightarrow \mathcal{E}(-1,0)  \xrightarrow{X^{\mathscr{A}}} \mathcal{E}^{\mathscr{A}} \xrightarrow{Z^a_{\mathscr{A}}}  \mathcal{E}^a\otimes_{\mathbb{C}}\mathcal{E}(-1,0)  \rightarrow 0 \\
& 0\rightarrow \mathcal{E}_a\otimes_{\mathbb{C}}\mathcal{E}(1,0) \xrightarrow{Z^a_{\mathscr{A}}} \mathcal{E}_{\mathscr{A}} \xrightarrow{X^{\mathscr{A}}} \mathcal{E}(1,0)\rightarrow 0 \\
& 0\rightarrow \mathcal{E}_{\overline{\alpha}}\otimes\mathcal{E}(0,1)\xrightarrow{Z^{\overline{\alpha}}_{\overline{A}}} \mathcal{E}_{\overline{A}} \xrightarrow{X^{\overline{A}}} \mathcal{E}(0,1)\rightarrow 0 \\
& 0\rightarrow \mathcal{E}_{\alpha}\otimes\mathcal{E}(1,0)\xrightarrow{Z^{{\alpha}}_{{A}}} \mathcal{E}_{{A}} \xrightarrow{X^{{A}}} \mathcal{E}(1,0)\rightarrow 0 \\
& 0 \rightarrow \mathcal{E}(-1,0) \xrightarrow{X^A} \mathcal{E}^A \xrightarrow{Z_A^{\alpha}} \mathcal{E}^{\alpha}\otimes\mathcal{E}(-1,0) \rightarrow 0 \\
& 0 \rightarrow \mathcal{E}(0,-1) \xrightarrow{X^{\overline{A}}} \mathcal{E}^{\overline{A}} \xrightarrow{Z_{\overline{A}}^{\overline{\alpha}}} \mathcal{E}^{\overline{\alpha}}\otimes\mathcal{E}(0,-1) \rightarrow 0.
\end{align*}
A choice of connection $\nabla\in \bD$ in the c-projective class determines a Weyl structure which splits these short exact sequences as follows
\begin{align*}
& 0\leftarrow \mathcal{E}(-1,0)  \xleftarrow{Y_{\mathscr{A}}} \mathcal{E}^{\mathscr{A}} \xleftarrow{W_a^{\mathscr{A}}}  \mathcal{E}^a\otimes_{\mathbb{C}}\mathcal{E}(-1,0)  \leftarrow 0 \\
& 0\leftarrow \mathcal{E}_a\otimes_{\mathbb{C}}\mathcal{E}(1,0) \xleftarrow{W_a^{\mathscr{A}}} \mathcal{E}_{\mathscr{A}} \xleftarrow{Y_{\mathscr{A}}} \mathcal{E}(1,0)\leftarrow 0 \\
& 0\leftarrow \mathcal{E}_{\overline{\alpha}}\otimes\mathcal{E}(0,1)\xleftarrow{W_{\overline{\alpha}}^{\overline{A}}} \mathcal{E}_{\overline{A}} \xleftarrow{Y_{\overline{A}}} \mathcal{E}(0,1)\leftarrow 0 \\
& 0\leftarrow \mathcal{E}_{\alpha}\otimes\mathcal{E}(1,0)\xleftarrow{W_{{\alpha}}^{{A}}} \mathcal{E}_{{A}} \xleftarrow{Y_{{A}}} \mathcal{E}(1,0)\leftarrow 0 \\
& 0 \leftarrow \mathcal{E}(-1,0) \xleftarrow{Y_A} \mathcal{E}^A \xleftarrow{W^A_{\alpha}} \mathcal{E}^{\alpha}\otimes\mathcal{E}(-1,0) \leftarrow 0 \\
& 0 \leftarrow \mathcal{E}(0,-1) \xleftarrow{Y_{\overline{A}}} \mathcal{E}^{\overline{A}} \xleftarrow{W^{\overline{A}}_{\overline{\alpha}}} \mathcal{E}^{\overline{\alpha}}\otimes\mathcal{E}(0,-1) \leftarrow 0.
\end{align*}
These {\em splitting tractors} $W$, $X$, $Y$, and $Z$ maps can be viewed as weighted tractors as follows
\begin{align*}
& W_{a}^{\mathscr{A}}\in\Gamma(\mathcal{E}^{\mathscr{A}}\otimes (\mathcal{E}_{a} \otimes_{\mathbb{C}}\mathcal{E}(1,0)) \ \ \ \ \ &&  W_{\overline{\alpha}}^{\overline{A}}\in\Gamma(\mathcal{E}_{\alpha}^{\overline{A}}(0,1))  &&& W_{\alpha}^{{A}}\in\Gamma(\mathcal{E}_{\alpha}^{{A}}(1,0))) \\
& X^{\mathscr{A}}\in\Gamma(\mathcal{E}^{\mathscr{A}}(1,0)) \ \ \ \ \ &&  X^{\overline{A}}\in\Gamma(\mathcal{E}^{\overline{A}}(0,1))  &&& X^{{A}}\in\Gamma(\mathcal{E}^{{A}}(1,0)) \\
& Y_{\mathscr{A}}\in\Gamma(\mathcal{E}^{\mathscr{A}}(-1,0)) \ \ \ \ \ &&  Y_{\overline{A}}\in\Gamma(\mathcal{E}^{\overline{A}}(0,-1))  &&& Y_{{A}}\in\Gamma(\mathcal{E}^{{A}}(-1,0)) \\
& Z^{a}_{\mathscr{A}}\in\Gamma(\mathcal{E}_{\mathscr{A}}\otimes(\mathcal{E}^{a}\otimes_{\mathbb{C}}\mathcal{E}(-1,0))) \ \ \ \ \ &&  Z^{\overline{\alpha}}_{\overline{A}}\in\Gamma(\mathcal{E}^{\alpha}_{\overline{A}}(0,-1))  &&& Z^{\alpha}_{{A}}\in\Gamma(\mathcal{E}^{\alpha}_{{A}}(-1,0)) 
\end{align*}
These maps satisfy the obvious relations
\begin{align*}
\begin{tabular}{ c| c  c }
 & $X^{\mathscr{A}}$ & $W^{\mathscr{A}}_{{a}}$ \\
\hline
$Y_{\mathscr{A}}$ & 1 & 0 \\ 
$Z_{\mathscr{A}}^{{b}}$ & 0 & $\delta^{{b}}_{{a}}$ \\ 
\end{tabular} \ \ \ \ \ \ \ \ 
\begin{tabular}{ c| c  c }
 & $X^{\overline{A}}$ & $W^{\overline{A}}_{\overline{\alpha}}$ \\
\hline
$Y_{\overline{A}}$ & 1 & 0 \\ 
$Z_{\overline{A}}^{\overline{\beta}}$ & 0 & $\delta^{\overline{\beta}}_{\overline{\alpha}}$ \\ 
\end{tabular} \ \ \ \ \ \ \ \ 
\begin{tabular}{ c| c  c }
 & $X^A$ & $W^A_{\alpha}$ \\
\hline
$Y_A$ & 1 & 0 \\ 
$Z_A^{\beta}$ & 0 & $\delta^{\beta}_{\alpha}$ \\ 
\end{tabular}
\end{align*}
Given two splittings (i.e., connections) $\nabla,\hat{\nabla}\in\bD$, sections of $\mathcal{E}^{\mathscr{B}}$ and $\mathcal{E}_{\mathscr{B}}$ change by
\begin{align}
&\widehat{\left(
\begin{array}{c}
\lambda^{b}\otimes_{\mathbb{C}} \sigma  \\
 \rho  \\
\end{array} 
\right)}
 = 
 \left(
\begin{array}{c}
\lambda^{b}\otimes_{\mathbb{C}} \sigma \\
\rho-\Upsilon_{b}\lambda^{b}\rho+\Upsilon_bJ^b_a\lambda^ai\rho  \\
\end{array} 
\right) \\
&\widehat{\left(
\begin{array}{c}
\upsilon  \\
\nu _{b}\otimes_{\mathbb{C}}\epsilon  \\
\end{array} 
\right)}
 = 
 \left(
\begin{array}{c}
\upsilon \\
\nu_{b}\otimes_{\mathbb{C}}\epsilon +\Upsilon_{b}\otimes_{\mathbb{C}} \epsilon\upsilon + J_b^a\Upsilon_a\otimes_{\mathbb{C}}\epsilon\upsilon i\\
\end{array} 
\right),
\end{align}
and sections of $\mathcal{E}^{{B}}$ and $\mathcal{E}_{{B}}$ change by
\begin{align}
&\widehat{\left(
\begin{array}{c}
\eta^{\beta}  \\
 \rho  \\
\end{array} 
\right)}
 = 
 \left(
\begin{array}{c}
\eta^{\beta} \\
\rho-2\Upsilon_{\beta}\eta^{\beta}  \\
\end{array} 
\right) \\
&\widehat{\left(
\begin{array}{c}
\xi  \\
\mu _{\beta}  \\
\end{array} 
\right)}
 = 
 \left(
\begin{array}{c}
\xi \\
\mu_{\beta} +2\Upsilon_{\beta}\xi  \\
\end{array} 
\right).
\end{align}

\subsection{The c-projective tractor connection}
For a choice of splitting, $\nabla\in\bD$, the tractor connection on $\mathcal{E}^{\mathscr{B}}$ and $\mathcal{E}_{\mathscr{B}}$ 
is given by
\begin{align}
&\nabla_a{\left(
\begin{array}{c}
\lambda^{b}\otimes_{\mathbb{C}} \sigma  \\
 \rho  \\
\end{array} 
\right)}
 = 
 \left(
\begin{array}{c}
(\nabla_a\lambda^{b})\otimes_{\mathbb{C}} \sigma + \lambda^{b}\otimes_{\mathbb{C}}(\nabla_a \sigma) + \delta_a^b\otimes_{\mathbb{C}}\rho \\
\nabla_a\rho-P_{ab}\lambda^{b}\sigma+P_{ab}J^b_c\lambda^ci\sigma  \\
\end{array} 
\right) \\
&\nabla_a{\left(
\begin{array}{c}
\upsilon  \\
\nu _{b}\otimes_{\mathbb{C}}\epsilon  \\
\end{array} 
\right)}
 = 
 \left(
 \begin{array}{c}
\nabla_a\upsilon -\nu_a\epsilon \\
(\nabla_a\nu _{b})\otimes_{\mathbb{C}}\epsilon + \nu _{b}\otimes_{\mathbb{C}}(\nabla_a\epsilon) +P_{ab}\otimes_{\mathbb{C}} \epsilon\upsilon - J_a^cP_{cb}\otimes_{\mathbb{C}}\epsilon\upsilon i\\
\end{array} 
\right),
\end{align}
and on sections of $\mathcal{E}^{{B}}$ and $\mathcal{E}_{{B}}$ it is given by
\begin{align}
&\nabla_{\alpha}{\left(
\begin{array}{c}
\eta^{\beta}  \\
 \rho  \\
\end{array} 
\right)}
 = 
 \left(
\begin{array}{c}
\nabla_{\alpha}\eta^{\beta}+\delta_{\alpha}^{\beta}\rho \\
\nabla_{\alpha}\rho-2P_{\alpha\beta}\eta^{\beta}  \\
\end{array} 
\right) \\
&\nabla_{\alpha}{\left(
\begin{array}{c}
\xi  \\
\mu _{\beta}  \\
\end{array} 
\right)}
 = 
 \left(
\begin{array}{c}
\nabla_{\alpha}\xi - \mu_{\alpha} \\
\nabla_{\alpha}\mu_{\beta} +2P_{\alpha\beta}\xi  \\
\end{array} 
\right)
\end{align}
\begin{align}
&\nabla_{\overline{\alpha}}{\left(
\begin{array}{c}
\eta^{\beta}  \\
 \rho  \\
\end{array} 
\right)}
 = 
 \left(
\begin{array}{c}
\nabla_{\overline{\alpha}}\eta^{\beta}\\
\nabla_{\overline{\alpha}}\rho-2P_{\overline{\alpha}\beta}\eta^{\beta}  \\
\end{array} 
\right) \\
&\nabla_{\overline{\alpha}}{\left(
\begin{array}{c}
\xi  \\
\mu _{\beta}  \\
\end{array} 
\right)}
 = 
 \left(
\begin{array}{c}
\nabla_{\overline{\alpha}}\xi  \\
\nabla_{\overline{\alpha}}\mu_{\beta} +2P_{\overline{\alpha}\beta}\xi  \\
\end{array} 
\right),
\end{align} 
where $P_{\alpha\beta}=\Pi^a_{\alpha}\Pi^b_{\beta}P_{ab}$ and $P_{\overline{\alpha}\beta}=\overline{\Pi}^a_{\overline{\alpha}}\Pi^b_{\beta}P_{ab}$. Using these formulae for the tractor connection a series of straightforward computations yield the following:
\begin{align*}
&\nabla_{\gamma}W_{\alpha}^{A}= -2P_{\gamma\alpha}X^A &&\nabla_{\overline{\gamma}}W_{\overline{\alpha}}^{\overline{A}}=-2P_{\overline{\gamma}\overline{\alpha}}X^{\overline{A}} \\
&\nabla_{\gamma}W_{\overline{\alpha}}^{\overline{A}}=-2P_{\overline{\gamma}\overline{\alpha}}X^{\overline{A}} && \nabla_{\overline{\gamma}}W_{\alpha}^A = -2P_{\overline{\gamma}\alpha}X^A 
\end{align*}
\begin{align*}
&\nabla_{\gamma}X^{A}= W_{\gamma}^A &&\nabla_{\overline{\gamma}}X^{\overline{A}}=W_{\overline{\gamma}}^{\overline{A}} \\
&\nabla_{\gamma}X^{\overline{A}}=0 &&\nabla_{\overline{\gamma}}X^A=0 
\end{align*}
\begin{align*}
& \nabla_{\gamma}Y_{A}= 2Z_A^{\alpha}P_{\gamma\alpha} && \nabla_{\overline{\gamma}}Y_{\overline{A}}= 2Z_{\overline{A}}^{\overline{\alpha}}P_{\overline{\gamma}\overline{\alpha}} \\
&  \nabla_{\gamma}Y_{\overline{A}}= 2Z_{\overline{A}}^{\overline{\alpha}}P_{\gamma\overline{\alpha}} && \nabla_{\overline{\gamma}}Y_{{A}}= 2Z_{{A}}^{{\alpha}}P_{\overline{\gamma}{\alpha}}
\end{align*}
\begin{align*}
& \nabla_{\gamma}Z_A^{\alpha} =-\delta^{\alpha}_{\gamma}Y_A && \nabla_{\overline{\gamma}}Z_{\overline{A}}^{\overline{\alpha}} =-\delta^{\overline{\alpha}}_{\overline{\gamma}}Y_{\overline{A}} \\
& \nabla_{\gamma}Z_{\overline{A}}^{\overline{\alpha}} =0 && \nabla_{\overline{\gamma}}Z_A^{\alpha} =  0
\end{align*} 
as well as
\begin{align*}
& \nabla_cW^{\mathscr{A}}_a= -P_{ab}X^{\mathscr{A}}+J^b_cP_{ba}iX^A \\
& \nabla_cX^{\mathscr{A}}= W^{\mathscr{A}}_c \\
& \nabla_cY_{\mathscr{A}}=  P_{ca}Z_{\mathscr{A}}^a-J_c^bP_{ba}iZ_{\mathscr{A}}^a \\
& \nabla_cZ_{\mathscr{A}}^a= -\delta^a_cY_{\mathscr{A}}.
\end{align*}

\subsection{The Thomas D-operator}
The Thomas $D$-operator $D_{\mathscr{A}}:\mathcal{E}^{\bullet}(w,w')\rightarrow \mathcal{E}_{\mathscr{A}}^{\bullet} (w-1,w')$ is a c-projectively invariant operator.  
%$D_{\overline{A}}:\mathcal{E}^{\bullet}(w,w')\rightarrow \mathcal{E}^{\bullet} (w,w'-1)$ 
Here $\mathcal{E}^{\bullet}$ denotes any tractor bundle constructed tensorially out of $\mathcal{E}^{\mathscr{A}}$, and $\mathcal{E}_{\mathscr{A}}$. For our purposes it will be sufficient to explicitly describe the action of the Thomas $D$-operator on sections $s$ of $\mathcal{E}(w,w')$. In a splitting the Thomas $D$-operator is given by
\begin{align*}
& D_{\mathscr{A}} s \colonequals \left(
\begin{array}{cc}
w s \\
\nabla_{a} s   \\
\end{array} 
\right)
 =Y_{\mathscr{A}} w s + Z^{a}_{\mathscr{A}} \nabla_{a} s.
 %& D_{\overline{A}} U^{\bullet} \colonequals \left(
%\begin{array}{cc}
%w' U^{\bullet} \\
%\nabla_{\overline{\alpha}} U^{\bullet}   \\
%\end{array} 
%\right)
% =Y_A w' U^{\bullet} + Z^{\overline{\alpha}}_A \nabla_{\overline{\alpha}}U^{\bullet},
\end{align*}

\subsection{The metricity bundle}
An {\em almost pseudo-Hermitian manifold} is a triple $(M,J,g)$ where $(M,J)$ is an almost complex manifold and $g$ is a Hermitian metric for $J$, i.e. $g_{ab}J^a_cJ^b_d=g_{cd}$. An almost pseudo-Hermitian manifold is called a {\em pseudo-Hermitain manifold} if $J$ is integrable. Recall the {\em K\"{a}hler form} of an almost pseudo-Hermitian manifold is the 2-form $\omega_{ab}=J^c_ag_{cb}$, which clearly satisfies $\omega_{ab}J^a_cJ^b_d=\omega_{cd}$ and $\omega_{ab}\omega^{bc}=-\delta^c_a$, where $\Omega^{bc}\colonequals J^c_a g^{ba}$ is the {\em Poisson bivector}. The almost Hermitian manifold is said to be {\em almost pseudo-K\"{a}hler} if the K\"{a}hler form, $\omega$, is closed. We can also view a Hermitian metric $g_{ab}$ as a real non-degenerate section $g_{\alpha\overline{\beta}} =\Pi_{\alpha}^a\overline{\Pi}^b_{\overline{\beta}}g_{ab}$ of $\mathcal{E}_{\alpha{\overline{\beta}}}$. 

Now we consider the bundle $\mathcal{E}^{A\overline{B}}$ and its real subbundles $\mathcal{H}^*$ (which, following \cite{CG4}, we term the {\em metricity bundle}) and its skew counterpart $\mathcal{P}^*$. $\mathcal{H}^*$  and $\mathcal{P}^*$ can also be viewed as subbundles of $
\mathcal{E}^{(\mathscr{AB})}$ and $\mathcal{E}^{ \mathscr{[AB]} }$, respectively, whose sections are Hermitian with respect to the almost complex structure, $J^{\mathscr{A}}_{\mathscr{B}}$, on the tractor bundle $\mathcal{E}^{\mathscr{A}}$. Note that $J^{\mathscr{A}}_{\mathscr{B}}$ gives an isomorphism between $\mathcal{H}^*$  and $\mathcal{P}^*$. In a splitting we have 
\begin{align*}
\mathcal{E}^{A\overline{B}} & = \mathcal{E}^{\alpha\overline{\beta}}(-1,-1) \oplus {\mathcal{E}^{\alpha}(-1,-1)} \oplus \mathcal{E}^{\overline{\beta}}(-1,-1) \oplus \mathcal{E}(-1,-1) \\
\mathcal{H}^*& =\operatorname{Herm}(T^*M)\otimes\mathcal{E}(-1,-1)_{\mathbb{R}} \oplus \mathcal{E}^a(-1,-1)_{\mathbb{R}} \oplus \mathcal{E}(-1,-1)_{\mathbb{R}}, \\
\mathcal{P}^*& =\operatorname{SkewHerm}(T^*M)\otimes\mathcal{E}(-1,-1)_{\mathbb{R}} \oplus \mathcal{E}^a(-1,-1)_{\mathbb{R}} \oplus \mathcal{E}(-1,-1)_{\mathbb{R}}
\end{align*}
where $\operatorname{Herm}(E)$ and $\operatorname{SkewHerm}(E)$ denotes the bundle of Hermitain metrics and Hermitian forms on a vector bundle $E$, respectively. We write sections $h^{A\overline{B}}\in \Gamma(\mathcal{E}^{A\overline{B}})$, $h^{\mathscr{AB}}\in\Gamma(\mathcal{H}^*)$, and $p^{\mathscr{AB}}\in\Gamma(\mathcal{P}^*)$ as 
\begin{align*}
h^{A\overline{B}}  = 
\left(
\begin{array}{cc}
\zeta^{\alpha\overline{\beta}} \\
\lambda ^{\alpha}   \ \ \  |   \ \ \   \mu^{\overline{\beta}}\\
\nu  \\
\end{array} 
\right) \ \ \ \ \ \ \ 
h^{\mathscr{A}\mathscr{B}}  = 
\left(
\begin{array}{cc}
\zeta^{ab} \\
\lambda ^{c}  \\
\nu  \\
\end{array} 
\right) \ \ \ \ \ \ \ 
p^{\mathscr{A}\mathscr{B}}  = 
\left(
\begin{array}{cc}
\pi^{ab} \\
\iota ^{c}  \\
\nu  \\
\end{array} 
\right)
\end{align*}
where we can identify the slots of $h^{\mathscr{A}\mathscr{B}}$ with real slots of $\mathcal{E}^{A\overline{B}}$: 
\begin{align*}
\overline{\ze^{\overline{\gamma}\beta}}=\ze^{\overline{\beta}\gamma}, \ \ \ \ \ \ \overline{\lambda^{\alpha}}=\mu^{\overline{\alpha}}, \ \ \ \ \ \ \operatorname{and} \ \ \ \ \ \  \overline{\nu}=\nu.
\end{align*}
We also see that the slots of $h^{\mathscr{A}\mathscr{B}}$ are related to the slots of $p^{\mathscr{A}\mathscr{B}}=J^{\mathscr{B}}_{\mathscr{C}}h^{\mathscr{A}\mathscr{C}}$ by
\begin{align*}
\pi^{ab}=J^b_c\ze^{ca},  \ \ \ \ \ \ \operatorname{and} \ \ \ \ \ \  \iota^{a}=J^a_b\lambda^{b}.
\end{align*}
We will also need to work with the dual bundles, namely $\mathcal{E}_{A\overline{B}}$ and its real subbundles $\mathcal{H}$ and $\mathcal{P}$. In a splitting these decompose into the following direct sums 
\begin{align*}
\mathcal{E}_{A\overline{B}} & = \mathcal{E}_{\alpha\overline{\beta}}(1,1) \oplus {\mathcal{E}_{\alpha}(1,1)} \oplus \mathcal{E}_{\overline{\beta}}(1,1) \oplus \mathcal{E}(1,1) \\
\mathcal{H} & =\operatorname{Herm}(TM)\otimes\mathcal{E}(1,1)_{\mathbb{R}} \oplus \mathcal{E}_a(1,1)_{\mathbb{R}} \oplus \mathcal{E}(1,1)_{\mathbb{R}}, \\
\mathcal{P} & =\operatorname{SkewHerm}(TM)\otimes\mathcal{E}(1,1)_{\mathbb{R}} \oplus \mathcal{E}_a(1,1)_{\mathbb{R}} \oplus \mathcal{E}(1,1)_{\mathbb{R}}.
\end{align*}
and we write sections $h_{A\overline{B}}\in \Gamma(\mathcal{E}_{A\overline{B}})$ and $h_{\mathscr{AB}}\in\Gamma(\mathcal{H}^*)$ as 
\begin{align*}
h_{A\overline{B}}  = 
\left(
\begin{array}{cc}
\phi_{\alpha\overline{\beta}} \\
\lambda _{\alpha}   \ \ \  |   \ \ \   \mu^{\overline{\beta}}\\
\tau  \\
\end{array} 
\right) \ \ \ \ \ \ \ 
h_{\mathscr{A}\mathscr{B}}  = 
\left(
\begin{array}{cc}
\phi_{ab} \\
\lambda _{c}  \\
\tau  \\
\end{array} 
\right)
\end{align*}
The formulae for the tractor connection applied to $\mathcal{H}^*$ and $\mathcal{H}$, respectively, are given by:
\begin{align*}
\nabla_{c}^{\mathcal{T}} {h^{\mathscr{A}\mathscr{B}}} & = 
\nabla_{c}^{\mathcal{T}}
\left(
\begin{array}{cc}
\zeta^{a{b}} \\
\lambda ^{a} \\
\nu  \\
\end{array} 
\right) 
= \left(
\begin{array}{cc}
\nabla_{c} \zeta^{a{b}} + \delta_{c}^{(a} \lambda^{{b)}} + J_c^{(a}J_i^{b)}\lambda^i \\
\nabla _{c} \lambda^{a} + 2\delta ^{a}_{c}\nu - 2P_{c{b}}\zeta^{a{b}}   \\
\nabla_{c} \nu - P_{cb}\lambda^{b}\\
\end{array} 
\right), \\
\nabla_{c}^{\mathcal{T}} {h_{\mathscr{A}\mathscr{B}}} & = 
\nabla_{c}^{\mathcal{T}}
\left(
\begin{array}{cc}
\tau  \\
\lambda _{a}\\
\phi_{ab} \\
\end{array} 
\right)
= \left(
\begin{array}{cc}
\nabla_{c} \tau - 2 \lambda _{c} \\
\nabla _{c} \lambda _{a} + P_{ca} \tau - \phi_{ca}\\
\nabla_{c} \phi_{ab} + 2P_{c(b} \lambda_{a)} +2P_{ci}\lambda_jJ^i_{(a}J^j_{b)} \\
\end{array} 
\right). 
\end{align*}

We can pass from sections of $h_{\mathscr{AB}}\in\Gamma(\mathcal{H}^*)$ to sections of $h_{A\overline{B}}\in \Gamma(\mathcal{E}_{A\overline{B}})$, or vice versa, via the $\Pi$ and $\overline{\Pi}$ maps discussed earlier. The formulae for the tractor connection applied to sections of $\mathcal{E}^{A\overline{B}}(-1,-1)$ are given by:
\begin{align*}
\nabla_{\gamma}^{\mathcal{T}} {h^{A\overline{B}}} & = 
\nabla_{\gamma}^{\mathcal{T}}
\left(
\begin{array}{cc}
\zeta^{\alpha\overline{\beta}} \\
\lambda ^{\alpha}\ \ \ | \ \ \ \mu^{\overline{\alpha}}\\
\nu  \\
\end{array} 
\right) 
= \left(
\begin{array}{cc}
\nabla_{\gamma} \zeta^{\alpha\overline{\beta}} + \delta_{\gamma}^{\alpha} \mu^{\overline{\beta}} \\
\nabla _{\gamma} \lambda^{\alpha} + 2\delta ^{\alpha}_{\gamma}\nu - 2P_{\gamma\overline{\beta}}\zeta^{\alpha\overline{\beta}}  \ \ \ | \ \ \  \nabla_{\gamma}\mu^{\overline{\beta}}-2P_{\gamma\beta}\ze^{\beta\overline{\alpha}}  \\
\nabla_{\gamma} \nu - P_{\gamma\overline{\alpha}} \mu ^{\overline{\alpha}} - P_{\alpha\beta}\lambda^{\beta}\\
\end{array} 
\right), 
\end{align*}
\begin{align*}
\nabla_{\overline{\gamma}}^{\mathcal{T}} {h^{A\overline{B}}} & = 
\nabla_{\overline{\gamma}}^{\mathcal{T}}
\left(
\begin{array}{cc}
\zeta^{\alpha\overline{\beta}} \\
\lambda ^{\alpha}\ \ \ | \ \ \ \mu^{\overline{\alpha}}\\
\nu  \\
\end{array} 
\right) 
= \left(
\begin{array}{cc}
\nabla_{\overline{\gamma}} \zeta^{\alpha\overline{\beta}} + \lambda^{{\alpha}}\delta_{\overline{\gamma}}^{\overline{\beta}}  \\
\nabla _{\gamma} \lambda^{\alpha}  - 2P_{\overline{\gamma}\overline{\beta}}\zeta^{\alpha\overline{\beta}}  \ \ \ | \ \ \  \nabla_{\overline{\gamma}}\mu^{\overline{\alpha}}+ 2\nu\delta_{\overline{\gamma}}^{\overline{\alpha}}-P_{\overline{\gamma}\beta}\ze^{\beta\overline{\alpha}}  \\
\nabla_{\overline{\gamma}} \nu - 2P_{\overline{\gamma}\overline{\alpha}} \mu ^{\overline{\alpha}} - 2P_{\overline{\gamma}\overline{\beta}}\lambda^{\beta}\\
\end{array} 
\right). \\
\end{align*}
The formulae for the tractor connection applied to sections of $\mathcal{E}_{A\overline{B}}(1,1)$ are given by:
\begin{align*}
\nabla_{\gamma}^{\mathcal{T}} {h_{A\overline{B}}} & = 
\nabla_{\gamma}^{\mathcal{T}}
\left(
\begin{array}{cc}
\phi_{\alpha\overline{\beta}} \\
\lambda_{\alpha}\ \ \ | \ \ \ \mu_{\overline{\alpha}}\\
\tau  \\
\end{array} 
\right) 
= \left(
\begin{array}{cc}
\nabla_{\gamma}\tau - \lambda_{\gamma} \\
\nabla _{\gamma} \lambda_{\alpha} +2 P_{\gamma\alpha}\tau   \ \ \ | \ \ \  \nabla_{\gamma}\mu_{\overline{\beta}} +2P_{\gamma\overline{\alpha}}\tau-\phi_{\gamma\overline{\alpha}} \\
\nabla_{\gamma} \phi_{\alpha\overline{\beta}}+ P_{\gamma\overline{\beta}}\lambda_{\alpha} + P_{\gamma\alpha}\mu_{\overline{\beta}}\\
\end{array} 
\right), 
\end{align*}
\begin{align*}
\nabla_{\overline{\gamma}}^{\mathcal{T}} {h_{A\overline{B}}} & = 
\nabla_{\overline{\gamma}}^{\mathcal{T}}
\left(
\begin{array}{cc}
\tau  \\
\lambda _{\alpha}\ \ \ | \ \ \ \mu_{\overline{\alpha}}\\
\phi_{\alpha\overline{\beta}} \\
\end{array} 
\right) 
= \left(
\begin{array}{cc}
\nabla_{\overline{\gamma}}\tau-\mu_{\overline{\gamma}}\\
\nabla _{\gamma} \lambda_{\alpha}  + 2P_{\overline{\gamma}\alpha}\tau-\phi_{\alpha\overline{\gamma}}  \ \ \ | \ \ \  \nabla_{\overline{\gamma}}\mu_{\overline{\alpha}} + 2P_{\overline{\alpha}\overline{\gamma}}\tau  \\
\nabla_{\overline{\gamma}} \phi_{\alpha\overline{\beta}} +P_{\overline{\gamma}\overline{\beta}}\lambda_{\alpha} + P_{\overline{\gamma}\alpha}\mu_{\overline{\beta}} \\
\end{array} 
\right). \\
\end{align*}

\subsection{C-projective BGG equations}\label{cprojbgg}

Given a Cartan geometry $(\mathcal{P}\twoheadrightarrow M, \omega)$ of type $(G,P)$ and a $G$-representation $\mathbb{V}$, we form a tractor bundle $\mathcal{V}=\mathcal{P}\times_P \mathbb{V}$. Then, via the corresponding tractor connection, we can form the exterior covariant derivative, $d^{\nabla}$, on
$\mathcal{V}$-valued forms to obtain the de Rham sequence twisted by
$\mathcal{V}$.
\[
0 \xrightarrow{} \mathcal{V} \xrightarrow{d^{\nabla}} \mathcal{V} \otimes \mathcal{E}_a \xrightarrow{d^{\nabla}} \mathcal{V} \otimes \mathcal{E}_{[ab]} \xrightarrow{d^{\nabla}} ...
\]
Then, via the canonical map
\[
\dagger : \mathcal{E}_a \rightarrow \operatorname{End}(\mathcal{V}), \ \ \ \ \operatorname{given\ explicitly\ by} \ \ \ \ \alpha_a \mapsto X^{\mathscr{B}}Z_{\mathscr{A}}^a \alpha_a
\]
in the case when $\mathcal{V} = \mathcal{E}^{\mathscr{A}}$ , one can construct a special case of the Kostant codifferential $\partial^{*}$, that gives a complex of natural bundle maps on $\mathcal{V}$-valued differential forms going in the opposite direction to the twisted de Rham sequence,
\[
0 \xleftarrow{\partial^{*}} \mathcal{V} \xleftarrow{\partial^{*}}  \mathcal{V} \otimes \mathcal{E}_a \xleftarrow{\partial^{*}}  \mathcal{V} \otimes \mathcal{E}_{[ab]} \xleftarrow{\partial^{*}}  ...
\]

The homology of this sequence gives natural subquotient bundles 
\begin{align*}
H_{k}(M,\mathcal{V})\colonequals \operatorname{ker}(\partial^{*})/\operatorname{im}(\partial^{*}).
\end{align*}
There are natural bundle projections $\Pi_{k}:
\operatorname{ker}(\partial^*) \subseteq \mathcal{V} \otimes
\mathcal{E}_{[ab...c]} \rightarrow H_{k}(M, \mathcal{V})$, from the
indicated $\mathcal{V}$-valued $k$-forms to the $k$th BGG homology.
Given a smooth section $\rho$ of $H_{k}(M, \mathcal{V})$ there is a
unique smooth section $L_{k}(\rho)$ of
$\operatorname{ker}(\partial^*)\subseteq \mathcal{V}\otimes
\mathcal{E}_{[a...b]}$ such that $\Pi_{k}(L_{k}(\rho))=\rho$ and
$\partial^{*}(d^{\nabla^{\mathcal{V}}}( L_{k}(\rho)))$$=0$.  This
characterizes a projectively invariant differential operator $L$
called the BGG splitting operator, or just the {\em splitting
  operator}. We can then define the $k$th BGG operator
$\Theta_{k}:H_{k}(M, \mathcal{V}) \rightarrow H_{k+1}(M, \mathcal{V})$
by $\rho \mapsto \Pi_{k+1}(d^{\nabla^{\mathcal{V}}}(L_{k}(\rho)))$. It
follows from these definitions that parallel sections of $\mathcal{V}$
are equivalent to (via $\Pi_0$ and $L_0$) a special class of so-called
{\em normal} solutions of the first BGG operator $\Theta_{0}:H_{0}(M,
\mathcal{V}) \rightarrow H_{1}(M, \mathcal{V})$ associated with
$\mathcal{V}$. Equations induced on the sections of $H_0(M,
\mathcal{V})$ by the BGG operator $\Theta_0$ are known as ({\em
  first}) {\em BGG equations}. Note that the BGG sequence, given by
the BGG operators, is not a complex in general, unless the connection
$\nabla^{\mathcal{V}}$ is flat. Note that a parallel section of a tractor bundle is necessarily in the image of the splitting operator. Next, we determine the first BGG equation and splitting operator corresponding to the c-projective metricity bundle.

\begin{proposition} \label{cproj3.2}
Let $(M, J, \bD)$ be an almost c-projective manifold. The first BGG operator $\Theta_0: H_0(M,\mathcal{H}^*) \rightarrow H_1(M,\mathcal{H}^*)$, induces the following projectively invariant first order equation on $\operatorname{Herm}(T^*M)\otimes \mathcal{E}(-1,-1)_{\mathbb{R}}$,
\begin{align}\label{cprojmeteq}
\nabla_{c} \zeta^{ab} - \frac{1}{m}\delta_{c}^{(a}\nabla_{d}\zeta^{b)d} - \frac{1}{m}J_{c}^{(b}J^{a)}_e\nabla_{d}\ze^{ed} = 0.
\end{align}
\end{proposition}

\begin{proof}
Let $h^{\mathscr{AB}}\in\Gamma(\mathcal{E}^{(\mathscr{AB})})$. Then we compute $\nabla_{c}^{\mathcal{T}} {h^{\mathscr{AB}}}$.
\begin{align*}
\nabla_{c}^{\mathcal{T}} {h^{\mathscr{AB}}} & = 
\nabla_{c}^{\mathcal{T}}
\left(
\begin{array}{cc}
\zeta^{ab} \\
\lambda ^{a}\\
\rho  \\
\end{array} 
\right) 
= \left(
\begin{array}{cc}
\nabla_{c} \zeta^{ab} + \delta_{c}^{(a} \lambda ^{b)} + J_c^{(a}J^{b)}_i\lambda^i\\
\nabla _{c} \lambda ^{a} + 2\delta ^{a}_{c}\rho - 2P_{cb}\zeta^{ab} \\
\nabla_{c} \rho - P_{ca} \lambda ^{a} \\
\end{array} 
\right). \\
\end{align*}
Then $\partial^*(\nabla_c^{\mathcal{T}}h^{\mathscr{AB}})=0$ tells us that the slots of $Z_{\mathscr{D}}^cX^{(\mathscr{A}}\nabla_c^{\mathcal{T}}h^{\mathscr{B})\mathscr{D}}$ are trace-free\footnote{See the proof of Proposition 14 of \cite{CG4} for a more details.}, i.e. we have the following system of equations:
\begin{align}\label{meteqalternateform}
\operatorname{trace}(\nabla_{c} \zeta^{ab} & + \delta_{c}^{(a} \lambda ^{b)} + J_c^{(a}J^{b)}_i\lambda^i )=0, \\
\operatorname{trace}(\nabla_{c} \lambda^{a} & - 2P_{cb} \zeta^{ab} + 2\delta^{a}_{c}\rho)=0. \nonumber
\end{align}
Therefore,
\begin{align*}
\lambda^{a} & = \frac{-1}{m}\nabla_{b} \zeta^{ab}, \\
\rho & = \frac{1}{2m} P_{ba} \zeta^{ab} + \frac{1}{4m^2} \nabla _{a} \nabla_{b} \zeta^{ab}.
\end{align*}
Thus a Hermitian form, $h$, on the cotractor bundle in the image of the splitting operator is of the form
\begin{align*}
{h^{\mathscr{AB}}} = L(\zeta^{ab}) = 
\left(
\begin{array}{cc}
\zeta^{ab} \\
\frac{-1}{m}\nabla_{b} \zeta^{ab} \\
\frac{1}{2m} P_{ba} \zeta^{ab} + \frac{1}{4m^2} \nabla _{a} \nabla_{b} \zeta^{ab} \\
\end{array} 
\right).
\end{align*}
Substituting gives the following first-order BGG equation on $\operatorname{Herm}(T^*M)\otimes \mathcal{E}(-1,-1)_{\mathbb{R}}$

\begin{align}
\operatorname{trace-free}(\nabla_c\ze^{ab}) = 0 \Longleftrightarrow
\nabla_{c} \zeta^{ab} - \frac{1}{m}\delta_{c}^{(a}\nabla_{d}\zeta^{b)d} - \frac{1}{m}J_{c}^{(b}J^{a)}_e\nabla_{d}\ze^{ed} = 0.
\end{align}
C-projective invariance follows from a straightforward computation. $H_0(M, \mathcal{H}^*) =\operatorname{Herm}(T^*M)\otimes \mathcal{E}(-1,-1)_{\mathbb{R}}$ follows from applications of the general BGG machinery of \cite{CSS} and this particular case is treated in Theorem 3.3 of \cite{CGH2}. So we have given the explicit form of $\Theta_0(\zeta^{ab})\colonequals \Pi_1(d^{\nabla}L(\zeta^{ab}))=0$, which is  the c-projective metrizability equation (\ref{cprojmeteq}).  
\end{proof}
Applying the procedure above to the bundle $\mathcal{H}$ yields $H_0(M,\mathcal{H})=\mathcal{E}(1,1)$ and also gives an explicit formula for the splitting operator $L:\mathcal{E}(1,1) \rightarrow \mathcal{H}$. For later reference, we give the formulae (cf. with Section 3.5 of \cite{CG4}) for the BGG splitting operators mapping into $\mathcal{H}^*$ and $\mathcal{H}$, respectively, in the following corollary,
\begin{corollary} \label{cproj3.3}
Let $\tau \in \Gamma (\mathcal{E}(1,1)_{\mathbb{R}})$ and $\ze \in \Gamma (\operatorname{Herm}(T^*M)\otimes \mathcal{E}(-1,-1)_{\mathbb{R}})$. Then their images under their respective splitting operators, both denoted by $L$, are given by
\begin{align*}
L(\tau)  &= 
\left(
\begin{array}{cc}
\tau  \\
\frac{1}{2}\nabla_{a}\tau \\
\frac{1}{2}(\delta^i_{(b}\delta^j_{c)}+ J^i_{(b}J^j_{c)})(\frac{1}{2}\nabla_{i}\nabla_{j}\tau + P_{ij}\tau) \\
\end{array} 
\right) 
\end{align*}
and
\begin{align*}
L(\zeta^{ab}) &= 
\left(
\begin{array}{cc}
\zeta^{ab} \\
\frac{-1}{m}\nabla_{i} \zeta^{ic}  \\
\frac{1}{2m} P_{ij} \zeta^{ij} + \frac{1}{4m^2} \nabla _{i} \nabla_{j} \zeta^{ij}  \\
\end{array} 
\right).
\end{align*}
\end{corollary}

\subsection{Determinants}
We now describe several methods of taking determinants which are relevant to our purposes. Let $\mathcal{\epsilon}_{\alpha_1\cdots\alpha_m}\in\Gamma(\mathcal{E}_{[\alpha_1\cdots\alpha_m]}(m+1,0))$
denote the canonical section giving the identification $\mathcal{E}^{[\alpha_1\cdots\alpha_m]} \xrightarrow{\sim} \mathcal{E}(m+1,0)$. Then we get a well-defined notion of determinant for sections of $\mathcal{E}^{\alpha\overline{\beta}}(k,k)$ via the map:
\begin{align*}
\bdet : \mathcal{E}^{\alpha\overline{\beta}}(k,k) & \rightarrow \mathcal{E}(km+m+1,km+m+1) \\
\sigma^{\alpha\overline{\beta}} & \mapsto \frac{1}{m!}{\epsilon}_{\alpha_1\cdots\alpha_m}\epsilon_{\overline{\beta}_1\cdots\overline{\beta}_m}\sigma^{\alpha_1\overline{\beta}_1}\cdots\sigma^{\alpha_m\overline{\beta}_m}.
\end{align*}
For Hermitian sections of $\mathcal{E}^{\alpha\overline{\beta}}(k,k)$, $\bdet$ is valued in $\mathcal{E}(km+m+1,km+m+1)_{\mathbb{R}}$. The parallel c-projectively invariant tractor
\begin{align*}
\epsilon_{A_0\cdots A_m\overline{B}_0\cdots\overline{B}_m}\colonequals {\epsilon}_{\alpha_1\cdots\alpha_m}\epsilon_{\overline{\beta}_1\cdots\overline{\beta}_m} Y_{[A_0}Z^{\alpha_1}_{A_1}\cdots Z^{\alpha_m}_{A_m]}Y_{[\overline{B}_0}Z^{\overline{\beta}_1}_{\overline{B}_1}\cdots Z^{\overline{\beta}_m}_{\overline{B}_m]},
\end{align*}
which is the {\em (complex) c-projective tractor volume form}, provides a method for taking determinants of sections of $\mathcal{E}^{A\overline{B}}$ as follows,
\begin{align*}
\operatorname{det}: \mathcal{E}^{A\overline{B}} & \rightarrow \mathcal{E}(0,0) \\
h^{A\overline{B}} & \mapsto \frac{1}{(m+1)!}\epsilon_{A_0\cdots\ A_m\overline{B}_0\cdots\ \overline{B}_m}h^{A_0\overline{B}_0}\cdots h^{A_m\overline{B}_m}.
\end{align*}
This determinant is real-valued, i.e., valued in $\mathcal{E}(0,0)_{\mathbb{R}}$, for Hermitian sections of $\mathcal{E}^{A\overline{B}}$. 

Now, letting $\epsilon_{a_1\cdots a_m} \colonequals \Pi^{\alpha_1}_{a_1} \cdots \Pi^{\alpha_m}_{a_m} \epsilon_{\alpha_1\cdots\alpha_m}$ and $\overline{\epsilon}_{b_1\cdots b_m} \colonequals  \overline{\Pi}^{\overline{\beta}_1}_{b_1} \cdots \overline{\Pi}^{\overline{\beta}_m}_{b_m} \epsilon_{\overline{\beta}_1\cdots\overline{\beta}_m}$, we define 
\[
\epsilon^2_{a_1\cdots a_m b_1 \cdots b_m}  \colonequals {\epsilon}_{[a_1\cdots a_m} \overline{\epsilon}_{b_1\cdots b_m]}  \in\Gamma(\mathcal{E}_{[a_1\cdots a_m b_1 \cdots b_m]}\otimes\mathcal{E}(m+1,m+1)_{\mathbb{R}}).
\]
So $\epsilon^2_{a_1\cdots a_m b_1 \cdots b_m}$ is the canonical section identifying oriented real line bundles $\mathcal{E}^{[a_1\cdots a_mb_1\cdots b_m]}\xrightarrow{\sim}\mathcal{E}(m+1,m+1)_{\mathbb{R}}$. Observe that this volume form gives a notion of determinant on sections of $\operatorname{SkewHerm}(T^*M)\otimes \mathcal{E}(k,k)_{\mathbb{R}}$ defined by,
\begin{align*}
\bdet : \operatorname{SkewHerm}(T^*M)\otimes \mathcal{E}(k,k)_{\mathbb{R}} & \rightarrow \mathcal{E}(km+m+1,km+m+1)_{\mathbb{R}} \\
\pi^{a{b}} & \mapsto \frac{1}{m!}\epsilon^2_{a_1\cdots a_m b_1 \cdots b_m} \pi^{a_1{b}_1}\cdots\pi^{a_m{b}_m}.
%\bdet : \operatorname{Herm}(T^*M)\otimes \mathcal{E}(k,k)_{\mathbb{R}} & \rightarrow \mathcal{E}(2km+2m+2,2km+2m+2)_{\mathbb{R}} \\
%\zeta^{a{b}} & \mapsto \frac{1}{m!}\epsilon^2_{a_1 \cdots a_{2m}}\epsilon^2_{b_1 \cdots b_{2m}} \zeta^{a_1{b}_1}\cdots\ze^{a_{2m}{b}_{2m}}.
\end{align*}
Since $J$ identifies $\operatorname{SkewHerm}(T^*M)$ with $\operatorname{Herm}(T^*M)$, we can pull the determinant back to $\operatorname{Herm}(T^*M)$. That is, let $\bdet (\sigma^{ab})\colonequals \bdet (\pi^{ab})$ where $\pi^{ab}=J^b_c\sigma^{ac}$. 
Then, we define the {\em (real) c-projective tractor volume form} by
\begin{align*}\label{realvolume}
\epsilon_{\mathscr{A}_0\cdots\ \mathscr{A}_m\mathscr{B}_0\cdots\ \mathscr{B}_m}\colonequals \epsilon^2_{a_1\cdots a_m b_1 \cdots b_m}   \Pi^{A_0}_{[\mathscr{A}_0}\cdots \overline{\Pi}^{\overline{B}_m}_{\mathscr{B}_m]} \Pi_{\alpha_1}^{a_1}\cdots \overline{\Pi}_{\overline{\beta}_m}^{b_m}Y_{[A_0}Z^{\alpha_1}_{A_1}\cdots Z^{\alpha_m}_{A_m]}Y_{[\overline{B}_0}Z^{\overline{\beta}_1}_{\overline{B}_1}\cdots Z^{\overline{\beta}_m}_{\overline{B}_m]}.
\end{align*}
Since $Y_{[A_0}Z^{\alpha_1}_{A_1}\cdots Z^{\alpha_m}_{A_m]}Y_{[\overline{B}_0}Z^{\overline{\beta}_1}_{\overline{B}_1}\cdots Z^{\overline{\beta}_m}_{\overline{B}_m]}$, and hence $\epsilon_{A_0\cdots A_m\overline{B}_0\cdots\overline{B}_m}$, is fixed under conjugation, it follows that \eqref{realvolume} is indeed a section of $(\Lambda_{\mathbb{R}}^{2m+2}\mathcal{T}^*)$. Thus it provides a notion of determinant on the real subbundle $\mathcal{P}^*$ of $\mathcal{E}^{A\overline{B}}$ via
\begin{align*}
\operatorname{det}: \mathcal{P^*} & \rightarrow \mathcal{E}(0,0)_{\mathbb{R}} \\
p^{\mathscr{AB}} & \mapsto \frac{1}{(m+1)!}\epsilon_{\mathscr{A}_0\cdots\ \mathscr{A}_m\mathscr{B}_0\cdots\ \mathscr{B}_m}p^{\mathscr{A}_0\mathscr{B}_0}\cdots p^{\mathscr{A}_m\mathscr{B}_m}.
\end{align*}
%where 
%\[
%\epsilon_{\mathscr{A}_0\cdots\ \mathscr{A}_m\mathscr{B}_0\cdots\ \mathscr{B}_m}\colonequals \epsilon^2_{{a}_1\cdots {a}_{m}{b}_1\cdots {b}_{m}}Y_{[\mathscr{A}_0}Z^{a_1}_{\mathscr{A}_1}\cdots Z^{a_m}_{\mathscr{A}_m}\tilde{Y}_{\mathscr{B}_0}\tilde{Z}^{b_1}_{\mathscr{B}_1}\cdots \tilde{Z}^{b_m}_{\mathscr{B}_m]},
%\]
%$\tilde{Y}_{\mathscr{B}}\colonequals J^{\mathscr{A}}_{\mathscr{B}}Y_{\mathscr{A}}$ and $\tilde{Z}^b_{\mathscr{B}}\colonequals J^{\mathscr{A}}_{\mathscr{B}}Z^b_{\mathscr{A}}$. 
Since $J$ (viewed as a complex structure at the tractor bundle level) identifies $\mathcal{P}^*$ with $\mathcal{H}^*$, we pull the determinant back to $\mathcal{H}^*$. That is, we let $\operatorname{det} (h^{\mathscr{AB}}) \colonequals \operatorname{det} (p^{\mathscr{AB}})$ where $p^{\mathscr{AB}}=J^{\mathscr{B}}_{\mathscr{C}}h^{\mathscr{AC}}$.

\subsection{Scalar curvature}

Let $(M,J)$ be an almost complex manifold equipped with an admissible Hermitian pseudo-Riemannian metric $g$. The volume form for $g$, $\operatorname{vol}_g\in\Gamma(\mathcal{E}(-2m-2))$, is parallel for any affine connection $\nabla$ preserving $g$ and hence any root of $\operatorname{vol}_g$ is parallel for $\nabla$ as well. In particular, $\tau\colonequals\operatorname{vol}_g^{-\frac{1}{m+1}}\in\Gamma(\mathcal{E}(1,1)_{\mathbb{R}})$ is parallel for the canonical connection $\nabla^g$ of $g$. It follows that $\ze^{ab}\colonequals \tau^{-1}g^{ab}\in\Gamma(\operatorname{Herm}(T^*M)\otimes \mathcal{E}(-2))$ is a solution to the metrizability equation. Further, in the splitting determined by the canonical connection for $g$ we see that 
\begin{align*}
L(\ze^{ab})  =
\left(
\begin{array}{cc}
\ze^{ab}  \\
0 \\
\frac{1}{2m}\ze^{ij}P_{ij} \\
\end{array} 
\right) =
\left(
\begin{array}{cc}
\tau^{-1}g^{ab}  \\
0 \\
\frac{\tau^{-1}}{2m}g^{ij}P_{ij} \\
\end{array} 
\right) =
\left(
\begin{array}{cc}
\tau^{-1}g^{ab}  \\
0 \\
\frac{\tau^{-1}R^g }{4m(m+1)}\\
\end{array} 
\right), 
\end{align*}
where $R^g$ denotes the scalar curvature of $g$. 
Then, up to a constant multiple, the determinant of $h^{\mathscr{AB}} = L(\tau^-1g^{ab})$ agrees with the scalar curvature $R^g$ of $g$, as was observed in Proposition 15 of \cite{CG4} and (in the special case of parallel sections of $\mathcal{E}^{A\overline{B}}$) in Proposition 4.8 of \cite{CEMN}. 
If $(M,J)$ can be realized as the interior of manifold with boundary $\overline{M}$ such that the c-projective structure of $\nabla^g$ admits a smooth extension to $\overline{M}$, then $\tau^{-1}g^{ab}$ and $R^g$ can be extended from the interior to $\overline{M}$, for details see Corollary 16 of \cite{CG4}. This is closely related to the questions we consider in section \ref{cprojresults}.

\section{Induced stratifications} \label{cprojresults}
\setcounter{theorem}{0}

We will show that given nondegeneracy of $L(\tau)$ or $L(\ze)$ where $\tau \in \Gamma (\mathcal{E}(1,1)_{\mathbb{R}})$ and $\ze \in \Gamma (\operatorname{Herm}(T^*M)\otimes \mathcal{E}(-1,-1)_{\mathbb{R}})$ is a solution to the metrizability equation \eqref{cprojmeteq} induces a stratification of the underlying almost c-projective manifold in analogous fashion to the projective cases considered in \cite{FloodGov}. The following theorems can be viewed as generalizations of curved orbit decomposition result of Theorem 3.3 in \cite{CGH2} where we are primarily using the more hands-on machinery developed in \cite{CG4}.

\begin{theorem}\label{cprojdensitynon-degenerate}
Let $(M,J,\bD)$ be an almost c-projective manifold with real dimension $2n$ equipped with a real density $\tau \in \Gamma ( \mathcal{E}(1,1)_{\mathbb{R}})$ such that $L(\tau)\in\Gamma(\mathcal{H})$ is non-degenerate as a Hermitian form on the tractor bundle. If $L(\tau)$ is definite then the zero locus $\mathcal{Z}(\tau)$ is empty and $(M,J,\bD,g)$ is Hermitian with metric $g_{bc}=(\delta^i_{(b}\delta^j_{c)}+ J^i_{(b}J^j_{c)})P_{ij}$, which is not, in general, admissible. If $L(\tau)$ has indefinite signature then $\mathcal{Z}(\tau)$ is either empty or it is a smoothly embedded separating real hypersurface such that the following hold: 
 \begin{enumerate}[(i)]
 \item $M$ is stratified by the strict sign of $\tau$ with curved orbit decomposition given by
$$
M=  \coprod\limits_{i \in \{+,0,-\}} M_i
$$
%%where $\tau >0$ on $M_+$  $\tau <0$ on $M_-$, and $\tau = 0$ on $M_0$. \\
 where $\tau$ is positive,
zero ,and negative on $M_{+}$, $M_{0}$, and $M_-$, respectively. 
\item If $M$ is closed, then the open components $(M \backslash M_{\mp},J,\bD)$ are 
c-projective compactifications of $(M_{\pm}, J, \nabla^{\tau})$, with boundary $M_0$.
\item The open components $(M \backslash M_{\mp},J,g)$ are pseudo-Hermitian with metric $g_{bc}=(\delta^i_{(b}\delta^j_{c)}+ J^i_{(b}J^j_{c)})P_{ij}$. The metric $g$ is not admissible for $\bD$ in general, but if $L(\tau)$ is parallel, then it is admissible and further $g$ is K\"{a}hler-Einstein.
\item $M_0$ inherits a (possibly degenerate) almost CR structure of hypersurface type. 
\end{enumerate}
\end{theorem}

\begin{proof}
(i) Since $L(\tau)$ is non-degenerate observe that $\mathcal{Z}(\tau)\cap\mathcal{Z}(\nabla\tau)=\varnothing$. The implicit function theorem implies then that $\mathcal{Z}(\tau)$ is a smoothly embedded real hypersurface, which is necessarily separating since $\nabla\tau\neq 0$ on $\mathcal{Z}(\tau)$. \\
(ii) Since $\tau$ is a defining density for $M_0$ it follows from Proposition \ref{cprojcompactnessprop} that $(M \backslash M_{\mp},J,\bD)$ are 
c-projective compactifications of $(M_{\pm}, J, \nabla^{\tau})$, with boundary $M_0$. \\
(iii) Away from $\mathcal{Z}(\tau)$, in the splitting $\nabla^{\tau}$, we see that $L(\tau) = (\tau,0,(\delta^i_{(b}\delta^j_{c)}+ J^i_{(b}J^j_{c)})P_{ij})$. Nondegeneracy of $L(\tau)$ implies that the Hermitian form $g_{bc}=(\delta^i_{(b}\delta^j_{c)}+ J^i_{(b}J^j_{c)})P_{ij}$ is itself non-degenerate, and hence a Hermitian metric, away from $\mathcal{Z}(\tau)$. If the bottom slot of $\nabla L(\tau)$ vanishes in the splitting $\nabla^{\tau}$, then $g$ is necessarily admissible and hence quasi-K\"{a}hler by our discussion in Section \ref{cprojcompactness}\footnote{ For details see Proposition 4.1 of \cite{CEMN} or Proposition 7 of \cite{CG4}.}.  In particular, if $L(\tau)$ is parallel, then $g$ is admissible. \\ 
(iv) 
%$(M,J,\bD)$ is a parabolic geometry of type $(SL(n+1,\mathbb{C}),P)$, for $P$ the stabilizer of a complex line in the standard representation $\mathbb{C}^{n+1}$ of $SL(n+1,\mathbb{C}$. The stabilizer of the Hermitian inner product $L(\tau)$ is a isomorphic to $SU(p+1,q+1)\subset SL(n+1,\mathbb{C})$. So the $M_0$ inherits a Cartan geometry of type $(SU(p+1,q+1),\overline{P})$ where $\overline{P}=SU(p+1,q+1)\cap P$ is the stabilizer of an isotropic line. As we saw in \cite{crparabolicsection}, these are precisely signature $(p,q)$ partially integrable almost CR structures of hypersurface type. 
Observe that, for $x \in M_0$, $H_x \colonequals T_xM_0 \cap J(T_xM_0)$ defines a corank one smooth distribution $H\subset TM_0$ and the pullback $i^*J$ of the almost complex structure along the inclusion $i:M_0 \hookrightarrow M$ defines an almost complex structure on $H$. Thus $(M_0^{2n-1},H,i^*J)$ is a (possibly degenerate) almost CR structure of hypersurface type. If $J$ is integrable then $i^*J$ is integrable. 
\end{proof}
%It follows from Lemma 5 of \cite{CG4} that for any defining function $\rho$ of $M_0$, the one form $\theta$ restricts to a contact form on $TM_0$, the two-form $d\theta$ restricted to the contact distribution $H$ is Hermitian. 
Next we examine an analogous result in the dual case. 

\subsection{Degenerate solutions of the c-projective metrizability equation: the order 2 c-projective compactification case}

\begin{theorem}\label{cprojone}
Let $(M,J,\bD)$ be a almost c-projective manifold with real dimension $2m$ equipped with a solution $\ze \in \Gamma (\operatorname{Herm}(T^*M)\otimes \mathcal{E}(-1,-1)_{\mathbb{R}})$ of the metrizability equation such that $L(\ze)\in\Gamma(\mathcal{H}^*)$ is non-degenerate as a pseudo-Hermitian form on the cotractor bundle. If $L(\ze)$ is definite then the degeneracy locus $\mathcal{D}(\ze)$ is empty and $(M, J, \bD,\zeta)$ is
  a quasi-K\"{a}hler manifold with inverse Hermitian metric $g^{-1}=\operatorname{sgn}(\tau)\tau
  \ze$ where $\tau={\normalfont \bdet}(\ze)$. If $L(\ze)$ has signature $(p+1,q+1)$, with $p,q\geq 0$, then $\mathcal{D}(\ze)$ is either empty or it is a smoothly embedded separating real hypersurface such that the following hold: 
 \begin{enumerate}[(i)]
\item $M$ is stratified by the strict signature of $\ze$ as a (density
  weighted) Hermitian form on $T^*M$ with curved orbit decomposition given by
$$
M=  \coprod\limits_{i \in \{+,0,-\}} M_i
$$
%%where $\tau >0$ on $M_+$  $\tau <0$ on $M_-$, and $\tau = 0$ on $M_0$. \\
 where $\zeta$ has signature $(p+1,q)$,
$(p,q+1)$,and $(p,q,1)$ on $TM$ restricted to $M_{+}$, $M_{-}$, and $M_0$, respectively. 
\item On $M_\pm$, $\zeta$ induces a quasi-K\"{a}hler metric $g_{\pm}$ with nonvanishing scalar curvature $R^{g_{\pm}}$, with the same signature as $\zeta$, with inverse $g^{-1}_{\pm} = \operatorname{sgn}(\tau)\tau \zeta|_{M_{\pm}}$ where $\tau={\normalfont \bdet}(\ze)$. 
\item If $M$ is closed, then the components $(M \backslash M_{\mp},J,\bD)$ are 
c-projective compactifications of $(M_{\pm}, J, \nabla^{\ze})$, with boundary $M_0$.
\item $M_0$ inherits a signature $(p,q)$ almost CR structure of hypersurface type.
\end{enumerate}
\end{theorem}

\begin{proof}
(i) Let $\Phi_{\mathscr{AB}} \colonequals (h^{\mathscr{AB}})^{-1}$ denote the pointwise inverse of $h^{\mathscr{AB}}=L(\zeta^{ab})$. Given a splitting, say $\nabla\in \bD$, we write $h$ and $\Phi$ as

\begin{align}
h^{\mathscr{AB}}  = 
\left(
\begin{array}{cc}
\ze^{ab}  \\
\lambda^a \\
\rho \\
\end{array} 
\right) 
\ \ \ \text{\normalfont and} \ \ \ 
{\Phi_{\mathscr{AB}}} = 
\left(
\begin{array}{cc}
\hat{\tau}  \\
\hat{\eta} _{a}\\
\hat{\phi}_{ab} \\
\end{array} 
\right),
\end{align}
for smooth sections $\rho \in \Gamma (\mathcal{E}(-1,-1)_{\mathbb{R}})$, $\lambda^a
\in \Gamma (\mathcal{E}^a(-1,-1)_{\mathbb{R}})$, $\ze^{ab} \in \Gamma
(\operatorname{Herm}(T^*M)\otimes\mathcal{E}(-1,-1)_{\mathbb{R}})$, $\hat{\tau} \in \Gamma (\mathcal{E}(1,1)_{\mathbb{R}})$, $\hat{\eta}_a
\in \Gamma (\mathcal{E}_a(1,1)_{\mathbb{R}})$, and $\hat{\phi}_{ab} \in \Gamma
(\operatorname{Herm}(TM)\otimes\mathcal{E}(1,1)_{\mathbb{R}})$. Observe that, up to a non-zero constant, $\hat{\tau}=\frac{\operatorname{\bdet}(\ze^{ab})}{\operatorname{det}(h^{\mathscr{AB}})}$.\\
By definition we have 
\begin{align}\label{kroneckerdelta}
\Phi_{\mathscr{AC}}h^{\mathscr{CB}} = \delta^{\mathscr{B}}_{\mathscr{A}}.
\end{align}
Applying the tractor connection, $\nabla^{\mathcal{T}}_{i}$, to both sides gives
\begin{align*}
-(\nabla^{\mathcal{T}}_{i} \Phi_{\mathscr{AC}})h^{\mathscr{CB}}=  \Phi_{\mathscr{AC}} \nabla^{\mathcal{T}}_{i} h^{\mathscr{CB}}.
\end{align*}
Applying $\Phi_{\mathscr{BD}}$ to each side gives
\begin{align}\label{three}
-\nabla^{\mathcal{T}}_{i} \Phi_{\mathscr{AD}} = \Phi_{\mathscr{AC}} (\nabla^{\mathcal{T}}_{i}h^{\mathscr{CB}}) \Phi_{\mathscr{BD}} .
\end{align}
Computing, the top two slots of \eqref{three} are given by
\begin{align*}
 2\hat{\eta}_i - \nabla_{i} \hat{\tau} & = \hat{\tau}^2\beta_i-2\hat{\tau}\hat{\eta}_b\alpha^b_i \\
 \hat{\phi}_{ic}-\nabla_i\hat{\eta}_c - P_{ic}\hat{\tau} & = \hat{\tau}\hat{\eta}_c\beta_i + \hat{\tau}\hat{\phi}_{jc}\alpha_i^j + \hat{\eta}_c\alpha_i^j\hat{\eta}_j
\end{align*}
where $\nabla_i^{\mathcal{T}}h^{\mathscr{CB}}=(0,\alpha_i^b,\beta_i)^t$. On $\mathcal{D}(\ze)=\mathcal{Z}(\hat{\tau})$ these reduce to 
\begin{align}\label{twoslots}
\nabla_{i} \hat{\tau} &= 2\hat{\eta}_i \\
\hat{\phi}_{ic} &= \hat{\eta}_c\alpha_i^j\hat{\eta}_j+ \nabla_i\hat{\eta}_c. 
\end{align}
It follows that, on $\mathcal{D}(\ze)$, $\Phi$ has the form
\begin{align}
{\Phi_{\mathscr{AB}}} = 
\left(
\begin{array}{cc}
\hat{\tau}  \\
\frac{1}{2}\nabla_c\hat{\tau}\\
\frac{1}{2}\nabla_a\nabla_b\hat{\tau} + \frac{1}{4}(\nabla_b\hat{\tau}) (\nabla_i\hat{\tau}) \alpha_a^i\\
\end{array} 
\right).
\end{align}

Nondegeneracy of $\Phi$ implies that $\mathcal{Z}(\hat{\tau})\cap\mathcal{Z}(\nabla\hat{\tau})=\varnothing$ whence we conclude that $\mathcal{Z}(\hat{\tau})$ is a smoothly embedded real hypersurface that is necessarily separating. 

(ii) On the open orbits $M_{\pm}$ the complex connection $\nabla^{\ze}\in \bD$ preserving $\ze$ necessarily preserves $\operatorname{\bdet}(\ze)$, and hence preserves  the pseudo-Riemannian metric $g^{-1}_{\pm} = \operatorname{sgn}(\tau)\tau \zeta|_{M_{\pm}}$. Since $\ze$ is Hermitian it follows that $g_{\pm}$ is Hermitian and hence admissible. As observed in Section \ref{admissiblemetrics}, a Hermitian pseudo-Riemannian metric, on an almost complex manifold, is admissible if and only if it is quasi-K\"{a}hler. \\
On $M_{\pm}$ in the splitting $\nabla^{\ze}$ we see that $L(\ze) = (\operatorname{sgn}(\tau)\tau^{-1}g^{-1}_{\pm}, 0, \frac{1}{4m(m+1)}\operatorname{sgn}(\tau)\tau^{-1}R^{g_{\pm}})^t$. Nondegeneracy of $L(\ze)$ implies that the scalar curvature $R^{g_{\pm}}$ is nonvanishing.  

(iii) Since the $(1,1)_{\mathbb{R}}$-density $\tau=\operatorname{\bdet}(\ze)$ is a defining density for $M_0$ and it is necessarily preserved by $\nabla^{\ze}$, it follows from Proposition \ref{cprojcompactnessprop} that the components $(M \backslash M_{\mp},J,\bD)$ are 
c-projective compactifications of $(M_{\pm}, J, \nabla^{\ze})$, with boundary $M_0$. 

(iv) Observe that, for $x \in M_0$, $H_x \colonequals T_xM_0 \cap J(T_xM_0)$ defines a corank one smooth distribution $H\subset TM_0$ and the pullback $i^*J$ of the almost complex structure along the inclusion $i:M_0 \hookrightarrow M$ defines an almost complex structure on $H$. Thus $(M_0^{2n-1},H,i^*J)$ is a (possibly degenerate) almost CR structure of hypersurface type. If $J$ is integrable then $i^*J$ is integrable.

Since $\nabla_a \tau$ is conormal to $TM_0\subset TM$ it follows that $J^a_b\nabla_a\tau$ is conormal to $J(TM_0) \subset TM$ so that $\nabla_a\tau \perp H$ and $J^a_b\nabla_a\tau \perp H$. 
%When $\ze$ becomes degenerate it has Hermitian corank 1 i.e real corank 2. So it's, Hermitian adjugate $\operatorname{Adj}(\ze)_{a_mb_m}\colonequals \epsilon_{a_1 \cdots a_m}\overline {\epsilon}_{b_1 \cdots b_m}\ze^{a_1b_1} \cdots \ze^{a_{m-1}b_{m-1}}$ has Hermitian rank 1, where $\overline{\epsilon}_{a \cdots b}\colonequals J^i_a \cdots J^j_b \epsilon_{i \cdots j}$.
 Now we show that when $\ze$ degenerates, its nullity is pointwise spanned by $\nabla_a \tau$ and $J_b^a\nabla_a \tau$. Since $\ze$ is Hermitian we need only show $\nabla_a\tau$ is in the nullity of $\ze$ and it will follow that $J_b^a\nabla_a \tau$ is in the nullity as well.

Now we show that there exists a scale such that $Y$ is null, i.e. $H^{\mathscr{A}\mathscr{B}}Y_{\mathscr{A}}Y_{\mathscr{B}}=0$, along $M_0$. Given $f$ a non-vanishing section of $\mathcal{E}(1,1)$, then in the scale corresponding to $f$, we have that 
\begin{align*}
f^{-1}D_{\mathscr{A}}f = f^{-1}Y_{\mathscr{A}}f + f^{-1}Z^{a}_\mathscr{A}\nabla_{a}f=Y_{\mathscr{A}}.
\end{align*}
Observing that, on $\mathcal{D}(\ze)$,
\begin{align*}
D_{\mathscr{A}}\hat{\tau} = Y_{\mathscr{A}}\hat{\tau} +  Z^{a}_{\mathscr{A}}\nabla_{a}\hat{\tau} = 2\Phi_{\mathscr{A}\mathscr{B}}X^{\mathscr{B}},
\end{align*}
This implies that along $\mathcal{D}(\ze)$
\begin{align*}
h^{\mathscr{A}\mathscr{B}}(D_\mathscr{A}\hat{\tau})(D_{\mathscr{B}}\hat{\tau}) = 4h^{\mathscr{A}\mathscr{B}}\Phi_{\mathscr{A}\mathscr{D}}X^{\mathscr{D}}\Phi_{\mathscr{B}\mathscr{C}}X^{\mathscr{C}}=4\Phi_{\mathscr{A}\mathscr{D}}X^{\mathscr{D}}X^{\mathscr{A}}=\hat{\tau}=0.
\end{align*}
Also, observe that
\begin{align*}
\frac{1}{2}h^{\mathscr{A}\mathscr{C}}D_{\mathscr{A}}\hat{\tau}=X^{\mathscr{C}}.
\end{align*}
Define $\xi \colonequals -\frac{1}{4}f\rho$ where $f$ is an arbitrary non-vanishing section of $\mathcal{E}(1,1)$ and $\rho \colonequals Y_\mathscr{A}Y_{\mathscr{B}} h^{\mathscr{A}\mathscr{B}}$. Now let $\gamma \colonequals f + \xi \hat{\tau}$. Then along $\mathcal{D}(\ze)$ we have
\begin{align*}
h^{\mathscr{A}\mathscr{B}}Y^{\gamma}_{\mathscr{A}}Y^{\gamma}_{\mathscr{B}} = &\ \gamma^{-2}(D_{\mathscr{A}}\gamma)({D_{\mathscr{B}}\gamma})h^{\mathscr{A}\mathscr{B}} \\
= &\ f^{-2}(D_{\mathscr{A}}f + \xi D_{\mathscr{A}}\hat{\tau})({D_{\mathscr{B}}f }+ \xi {D_{\mathscr{B}}\hat{\tau} }) h^{\mathscr{A}\mathscr{B}} \\
 = &\ f^{-2}(D_\mathscr{A}f)({D_{\mathscr{B}}f})h^{\mathscr{AB}} + f^{-2}\xi (D_{\mathscr{A}}\hat{\tau})({D_{\mathscr{B}}f})h^{\mathscr{AB}} + f^{-2}\xi (D_{\mathscr{A}}f)({D_{\mathscr{B}}\hat{\tau}})h^{\mathscr{AB}} \\
 &\ + f^{-2}\xi^{2}(D_{\mathscr{A}}\hat{\tau})({D_{\mathscr{B}}\hat{\tau}})h^{\mathscr{AB}} \\
 = &\ Y_{\mathscr{A}}Y_{\mathscr{B}}h^{\mathscr{A}\mathscr{B}} + 2f^{-2}\xi X^{\mathscr{B}}({D_{\mathscr{B}}f}) + 2f^{-2}\xi X^{\mathscr{A}}(D_{\mathscr{A}}f) + 0 \\
= &\ \rho + 4f^{-1}\xi \\
= &\ \rho - \rho \\
= &\ 0.
\end{align*}
A scale $\nabla^\gamma$ preserving $\gamma$ such that $Y^{\gamma}$ is null along along $\mathcal{D}(\ze)$ will be known as a {\em special boundary scale}.\footnote{This term and the method of constructing a special boundary scale was first done the setting of projective differential geometry by Sam Porath in \cite{GovPorath}.} In such a scale, $\rho$ necessarily vanishes along $\mathcal{D}(\ze)$.

Computing the slots of \ref{kroneckerdelta} along $\mathcal{D}(\ze)$ in a special boundary scale yields the following system of equations
\begin{align}
\label{firstone} & \frac{1}{2}(\nabla_a\hat{\tau})\ze^{ab} =0 \\ 
\label{secondone} &\frac{1}{2}(\nabla_c\nabla_a\hat{\tau})\nabla_j\ze^{jc}  +\frac{1}{4}(\nabla_a\hat{\tau})(\nabla_i\hat{\tau})\alpha_c^i\nabla_j\ze^{jc} = 0 \\
\label{thirdone} &\frac{1}{2m}(\nabla_c\hat{\tau})(\nabla_i\ze^{ic}) =1 \\
\label{fourthone} &\frac{1}{2m}(\nabla_a\hat{\tau})(\nabla_i\ze^{ib}) + \frac{1}{2}(\nabla_c\nabla_a\hat{\tau})\ze^{cb} +\frac{1}{4}(\nabla_a\hat{\tau})(\nabla_i\hat{\tau})\alpha_c^i\nabla_j\ze^{jc}  = \delta^b_a.
\end{align}

It follows from \eqref{firstone} that the kernel of $\ze$ on $\mathcal{D}(\ze)$ is spanned by $\nabla_a\hat{\tau}$ and $J_b^a\nabla_a\hat{\tau}$. So $\ze$ induces a pseudo Hermitian metric with signature $(p,q)$ on the distribution $H$. Contracting $\xi^a \in \Gamma(H)$ into \eqref{secondone} gives us

\begin{align}\label{reebeq} 
\frac{1}{2}(\nabla_c\nabla_a\hat{\tau})(\nabla_j\ze^{jc})\xi^a=0.
\end{align} 

The restriction of $\theta_a\colonequals J_a^i\nabla_i\hat{\tau}$ to $TM_0$ is a weighted contact form for the CR structure induced on $\mathcal{D}(\ze)$. Then \eqref{reebeq} together with \eqref{thirdone} this implies that $T^a\colonequals-\frac{1}{2m}J^a_j\nabla_i\ze^{ij}$ is a candidate for the Reeb vector field since $\theta_a T^a=1$. 

Contracting $\xi^a\in \Gamma(H)$ into \eqref{fourthone} gives us
\begin{align*}
\frac{1}{2}(\nabla_c\nabla_a\hat{\tau})\ze^{cb}\xi^a=\delta^b_a\xi^a=\xi^b,
\end{align*}
whence $(\ze|_H)^{-1}=\frac{1}{2}\nabla\nabla\hat{\tau}$. Since  and the real part of the CR Levi form corresponding to $\theta_a$ is the restriction of $(d\theta)_{ab}=2\nabla_{[a}J_{b]}^i\nabla_i\hat{\tau}$ to $H$, it follows immediately that the real part of the CR Levi form is $(J^a_b\check{\ze}^{bc})^{-1}$, where $\check{\ze}\colonequals \ze|_H$. Thus the CR structure is Levi-non-degenerate with signature $(p,q)$ and, since $T^a(d\theta)_{ab}=0$, we see that $T^a$ is indeed the Reeb. 
\end{proof}

\begin{remark}
If the original almost complex structure is integrable then the open orbits are in fact pseudo-K\"{a}hler and the c-projective Schouten tensor $P_{ab}$ is symmetric and Hermitian, so by Theorem 23 of \cite{CG4}, $g_{\pm}$ satisfies an asymptotic version of the Einstein equation. If $L(\ze)$ is parallel then the vanishing of the middle slot of $\nabla L(\ze)$ implies that the quasi-K\"{a}hler metrics on the open orbits are Einstein. If $J$ is integrable and $L(\ze)$ is parallel then the open orbits are pseudo-K\"{a}hler-Einstein.
\end{remark}

\begin{corollary}
Let $(\overline{M},J)$ be a connected almost complex manifold with boundary $\partial M$ and interior $M$, equipped with a pseudo-quasi-K\"{a}hler metric $g$ on $M$ which is Hermitian for $J$, has nonvanishing scalar curvature, and such that the minimal complex connection $\nabla^{g}$  preserving $g$ does not extend to any neighborhood of a boundary point, but the almost c-projective structure $\bD \colonequals [\nabla^g]$ does extend to the boundary. Let $\tau \colonequals \operatorname{vol}(g)^{-\frac{1}{m+1}}$. Then $\ze^{ab}\colonequals \tau^{-1}g^{ab}$ extends to the boundary. If $L(\ze^{ab})$ is non-degenerate on $\overline{M}$, then $(M,J,\nabla^{g})$ is c-projectively compact.
\end{corollary}

\begin{proof}
$L(\zeta)$ is defined on the interior $M$. It extends via parallel transport for the prolongation connection\footnote{The prolongation connection is a natural modification of the tractor connection. See, e.g., \cite{HSSS} for a construction for general BGG operators} to a parallel (for the prolongation connection) tractor on $\overline{M}$. Projection to the quotient bundle $\operatorname{Herm}(T^*M) \otimes \mathcal{E}(-1,-1)_{\mathbb{R}}$ gives a smooth extension of $\zeta$ to all of $\overline{M}$ ({\em cf.} Corollary 16 of \cite{CG4}). The degeneracy locus of $\zeta$ is precisely $\partial M$, otherwise it would contradict our assumption that $\nabla^g$ does not extend to any neighborhood of a boundary point. Then the result follows from Theorem \ref{cprojone}.
\end{proof}

\subsection{The Model}\label{model3}
We briefly discuss here the model for the structures considered in
Theorem \ref{cprojone}. The standard homogeneous model for c-projective
geometry is the complex projective space arising as the complex scalar projectivization,
$\mathbb{CP}^{m}=\mathbb{P}(\mathbb{C}^{m+1})$, of $\mathbb{C}^{m+1}$. The J-planar curves in $\mathbb{CP}^m$
are the smooth (real) curves lying in linearly embedded complex curves $\mathbb{CP}^1 \hookrightarrow \mathbb{CP}^m$. On $\mathbb{CP}^m$ the group
$G=SL(\mathbb{C}^{m+1})$ acts transitively. On this c-projective
structure it is well known that the tractor connection is induced by the trivial connection on $\mathbb{C}^{m+1}$. Now suppose we fix, on
$\mathbb{C}^{m+1}$, a non-degenerate symmetric bilinear form $h$ of
signature $(p+1,q+1)$. This may be identified with a corresponding
parallel tractor.  In $G$ consider the subgroup $H:=SU(h)\cong
SU(p+1,q+1)$ fixing $h$ (so $p+q=m-1$). This acts on the complex projective space $\mathbb{CP}^m$ but now with orbits parametrized by the strict sign of
$h(X,X)$ where $X$ denotes the homogeneous co\"{o}rdinates of a given
point on $\mathbb{CP}^m$. Complex projective space $\mathbb{CP}^m$ equipped with this action
of $H$ and accompanying orbit decomposition is the model for the
structure discussed in Theorem \ref{cprojone}. This follows easily from the
tractor approach that we use with the interpretation of the tractor
bundles over the homogeneous space $G/P$. So the Theorem also
reveals, for this model, the general features of the orbits and the
geometries thereon. In fact, $h^{-1}=L(\zeta)$ where $\zeta$ is the
corresponding solution of (\ref{cprojmeteq}) and, in the language of
\cite{CSlov}, this is a holonomy reduction of a flat Cartan geometry
(namely $G\rightarrow \mathbb{CP}^m$). Turning this around, we see that
Theorem \ref{cprojone} shows that solutions $\zeta$ of equation
(\ref{cprojmeteq}), satisfying that $\det (L(\zeta)) $ is nowhere zero,
provide well behaved curved generalizations of this model even though
$\zeta$ is not required to be normal (i.e. $L(\zeta)$ is not required
to be parallel).


\begin{thebibliography}{99}


\bibitem{AGW}  C. Arias, A. R. Gover, A. Waldron. {\em Conformal geometry of embedded manifolds with boundary from universal holographic formulae},  Adv. Math. 284 (2021),  107700, 75 pp.
  
\bibitem{BEGover} T.N. Bailey, M.G. Eastwood, A.R. Gover. {\em Thomas's structure bundle for conformal,
projective and related structures}. Rocky Mountain J. Math., 24 (1994), 1191--1217.

\bibitem{BEGraham} T.N. Bailey, M.G. Eastwood, C.R. Graham. {\em Invariant theory for conformal and CR
geometry, Annals of Mathematics}. 139 (1994), 491--552.

\bibitem{BL} H. Baum, F. Leitner. {\em The geometric structure of Lorentzian manifolds with twistor spinors in low dimension}. Dirac operators: yesterday and today, 229--240, Int. Press, Somerville, MA, 2005.

\bibitem{Bochner} S. Bochner. {\em Curvature in Hermitian metric}. Bull. Amer. Math. Soc.53 (1947), 179--195.

\bibitem{BM} A.V. Bolsinov, V.S. Matveev. {Local normal forms for geodesically equivalent pseudo-Riemannian metrics}. Trans. Amer. Math. Soc. 367 (2015), no. 9, 6719--6749. 

\bibitem{BMMR} A. V. Bolsinov, V. S. Matveev, T. Mettler, S. Rosemann. {\em Four-dimensional K\"{a}hler metrics admitting c-projective vector fields}. J. Math. Pures Appl. (9) 103 (2015), no. 3, 619--657.

\bibitem{BCEG} T. Branson, A. \v{C}ap, M. Eastwood, A. R.  Gover. {\em Prolongations of geometric overdetermined systems}. Internat. J. Math. 17 (2006), no. 6, 641--664. 

\bibitem{BransonGover} T. Branson, A.R. Gover. {\em Conformally invariant non-local operators} Pacific Journal of Mathematics, 201 (2001), 19--60.

\bibitem{BDE} R. Bryant, M. Dunajski, M. Eastwood. 
{\em Metrisability of two-dimensional projective structures.} J. Differential Geom. 83 (2009), no. 3, 465--499. 

\bibitem{CD} D. Calderbank, T.\ Diemer. {\em Differential invariants and curved Bernstein-Gelfand-Gelfand sequences}. J. Reine Angew. Math. 537 (2001), 67--103.
  
\bibitem{CEMN} D. M. J. Calderbank, M. G. Eastwood, V. S. Matveev, K. Neusser.  {\em C-projective geometry}. Mem. Amer. Math. Soc. 267 (2020), no. 1299, v+137.

\bibitem{CGconformal} A. \v{C}ap, A.R. Gover. {\em Standard tractors and the conformal ambient metric construction} Annals Global Anal.Geom., 24 (2003), 231--259.

\bibitem{CG3} A. \v{C}ap, A. R. Gover.
{\em Scalar curvature and projective compactness}.  J. Geom. Phys. 98 (2015), 475--481.

\bibitem{CG4} A. \v{C}ap, A. R. Gover. {\em C-projective compactification; (quasi)--K\"{a}hler metrics and CR boundaries}, Amer. J. Math., 141, no. 3 (2019) 813--856. 

\bibitem{CG1} A. \v{C}ap, A. R. Gover.
{\em Projective compactifications and Einstein metrics}. J. Reine Angew. Math. 717 (2016), 47--75. 

\bibitem{CG2} A. \v{C}ap, A. R. Gover.
{\em Projective compactness and conformal boundaries}. Math. Ann. 366 (2016), no. 3-4, 1587--1620.

\bibitem{CGH1} A. \v{C}ap, A. R. Gover, M. Hammerl.
{\em Projective BGG equations, algebraic sets, and compactifications of Einstein geometries.} J. Lond. Math. Soc. (2) 86 (2012), no. 2, 433--454. 

\bibitem{CGH2} A. \v{C}ap, A. R. Gover, M. Hammerl.
{\em Holonomy reductions of Cartan geometries and curved orbit decompositions,} Duke Math. J. 163 (2014), no. 5, 1035--1070.

\bibitem{CGM} A. \v{C}ap, A. R. Gover, H. R. Macbeth.
{\em Einstein metrics in projective geometry}. Geom. Dedicata 168 (2014), 235--244. 

\bibitem{CSlov} A. \v{C}ap, J. Slov\'{a}k.
{\em Parabolic geometries. I. Background and general theory.} Mathematical Surveys and Monographs, 154. American Mathematical Society, Providence, RI, 2009. x+628 pp.

\bibitem{CSS} A. \v{C}ap, J. Slov\'{a}k, V. Sou\v{c}ek. {\em Bernstein-Gelfand-Gelfand sequences}. Ann. of Math. (2) 154 (2001), no. 1, 97--113.

\bibitem{CSouc} A. \v{C}ap, V. Sou\v{c}ek.
{\em Curved Casimir operators and the BGG machinery.} SIGMA Symmetry Integrability Geom. Methods Appl. 3 (2007), Paper 111, 17 pp. 

\bibitem{CurGov} S. Curry, A. R. Gover.  {\em An introduction to conformal geometry and tractor calculus, with a view to applications in general relativity.} Asymptotic Analysis in General Relativity, Eds. Daud\'e, H\"{a}fner, Nicolas, LMS, CUP, 2018.

\bibitem{CurryCR} S. N. Curry, A. R. Gover {\em CR Embedded Submanifolds of CR Manifolds}. Memoirs of the American Math Society, 258 (2019) no. 1241.

\bibitem{Derd} A. Derdzinski.  {\em Zeros of conformal fields in any metric signature.} Classical Quantum Gravity 28 (2011), no. 7, 075011, 23 pp.

\bibitem{DE} M. Dunajski, M. Eastwood. 
{\em Metrisability of three-dimensional path geometries.} Eur. J. Math. 2 (2016), no. 3, 809--834.

\bibitem{EM} M. Eastwood, V. Matveev
{\em Metric connections in projective differential geometry}. Symmetries and overdetermined systems of partial differential equations, 339--350, IMA Vol. Math. Appl., 144, Springer, New York, 2008.

\bibitem{FM} A. Fedorova, V.S. Matveev. {\em Degree of mobility for metrics of Lorentzian signature and parallel (0,2)-tensor fields on cone manifolds.} Proc. Lond. Math. Soc. (3) 108 (2014), no. 5, 1277--1312.

\bibitem{FG} C. Fefferman, C. R. Graham, {\em The ambient metric}. Annals of Mathematics Studies, 178. Princeton University Press, Princeton, NJ, 2012. x+113 pp. 

\bibitem{FloodGov} K. J. Flood, A. R. Gover, {\em Metrics in projective differential geometry: the geometry of solutions to the metrizability equation}. J. Geom. Anal. (2019) 29 (3): 2492-2525. 

\bibitem{Frost} G. E. Frost. {\em The Projective Parabolic Geometry of Riemannian, K\"{a}hler and Quaternion-K\"{a}hler Metrics}. arXiv:1605.04406

\bibitem{GoverGraham} A. R. Gover, C. R. Graham. {\em CR Invariant powers of the sub-Laplacian}. J.
Reiene Angew. Math. 583 (2005), 1--27.

\bibitem{GLW} A. R. Gover, E. Latini, A. Waldron. {\em  Poincare-Einstein holography for forms via conformal geometry in the bulk.} Mem. Amer. Math. Soc. 235 (2015), no. 1106, vi+95 pp. 

\bibitem{GLW2} A. R. Gover, E. Latini, A. Waldron. {\em Metric projective geometry, BGG detour complexes and partially massless gauge theories}.  Comm. Math. Phys. 341 (2016), no. 2, 667--697.
  
\bibitem{GH} A.R. Gover, H.R. Macbeth. {\em Detecting Einstein geodesics: Einstein metrics in projective and conformal geometry}. Differential Geometry and Its Applications, 33 (2014), 44--69 
  
\bibitem{GM} A.R. Gover, V. Matveev. {\em Projectively related metrics, Weyl nullity, and metric projectively invariant equations.} Proc. Lond. Math. Soc. (3) 114 (2017), no. 2, 242--292. 

\bibitem{GNW} A.R. Gover, K. Neusser, T. Willse.  {\em Projective geometry of Sasaki-Einstein structures and their compactifications}. Dissertationes Mathematicae 546 (2019), 64 pp.
  
\bibitem{GPW} A.R. Gover, R. Panai, T. Willse. {\em Nearly K\"{a}hler geometry and $(2,3,5)$-distributions via projective holonomy}. Indiana Univ. Math. J. 66 No. 4 (2017), 1351--1416

\bibitem{GovPorath} A. R. Gover, S. Porath. {\em A work in progress on projective hypersurfaces}.

\bibitem{GZ} C.R. Graham, M. Zworski, {\em Scattering matrix in conformal geometry}. Invent. Math. 152 (2003), no. 1, 89--118.

\bibitem{GrayHervella} A. Gray, L. M. Hervella. {\em The sixteen classes of almost Hermitian manifolds and their linear invariants}. Ann. Mat. Pura Appl. 123 (1980), 35--58.

\bibitem{Jez} J. Jezierski. {\em Conformal Yano-Killing tensors in anti-de Sitter spacetime}. Classical Quantum Gravity 25 (2008), no. 6, 065010, 17 pp.

\bibitem{Juh}A. Juhl.  Families of conformally covariant differential operators, Q-curvature and holography. Progress in Mathematics, 275. Birkhuser Verlag, Basel, 2009. xiv+488 pp. 

\bibitem{Z3} J. Jung, S. Zelditch. {\em Number of nodal domains and singular points of eigenfunctions of negatively curved surfaces with an isometric involution}. J. Differential Geom. 102 (2016), no. 1, 37--66.

\bibitem{KW1}J.L. Kazdan, F.W. Warner. {\em Existence and conformal deformation of metrics with prescribed Gaussian and scalar curvatures.} Ann. of Math. (2) 101 (1975), 317--331.

\bibitem{KW2}J.L. Kazdan, F.W. Warner.{\em Scalar curvature and conformal deformation of Riemannian structure.} J. Differential Geometry 10 (1975), 113--134.

\bibitem{Lichnerowicz} A. Lichnerowicz. {\em Th\'{e}orie globale des connexions et des groupes d'holonomie} Roma, Edizioni Cremonese (1955).

\bibitem{KruglikovThe} B. Kruglikov, V.S. Matveev, D. The. {\em Submaximally symmetric c-projective structures}. Int. J. Math. 27 (2016), 1650022, 34 pp.

\bibitem{KMS} M. Kol\'{a}\v{r}, P. Michor, J. Slov\'{a}k. {\em Natural Operations in Differential Geometry}. Springer-Verlag. Berlin, Heidelberg. 1993. x+434 pp.

\bibitem{KM} B. Kruglikov, V.S. Matveev. {\em The geodesic flow of a generic metric does not admit nontrivial integrals polynomial in momenta}. Nonlinearity 29 (2016), no. 6, 1755--1768.

\bibitem{Leit2} F. Leitner. {\em About twistor spinors with zero in Lorentzian geometry}. SIGMA Symmetry Integrability Geom. Methods Appl. 5 (2009), Paper 079, 12 pp.

\bibitem{Lis} A. Lischewski, {\em The zero set of a twistor spinor in any metric signature}. Rend. Circ. Mat. Palermo (2) 64 (2015), no. 2, 177--201.

\bibitem{Mar} T. Marugame, {\em  Volume Renormalization for the Blaschke Metric on Strictly Convex Domains}. J. Geom. Anal. 28 (2018), no. 1, 510--545.

\bibitem{Mat1} V. S. Matveev, {\em Projective Lichnerowicz-Obata conjecture}. J. Diff. Geom. 75 (2007), 459-502. 
  
\bibitem{Mat} V.S. Matveev. {\em On projective equivalence and pointwise projective relation of Randers metrics.} Internat. J. Math. 23 (2012), no. 9, 1250093, 14 pp.

\bibitem{Mat2} V. S. Matveev. {\em Geodesically equivalent metrics in general relativity.} J. Geom. Phys. 62 (2012), no. 3, 675--691.

\bibitem{Mat5} V. S. Matveev, S. Rosemann. {\em Proof of the Yano-Obata Conjecture for holomorph-projective transformations}. J. Diff. Geom. 92(2012) 221-261.

\bibitem{MR} V. S. Matveev, S. Rosemann. {\em The degree of mobility of Einstein metrics}. J. Geom. Phys. 99 (2016), 42--56.

\bibitem{MM} R. R. Mazzeo, R. B. Melrose.{\em Meromorphic extension of the resolvent on complete spaces with asymptotically constant negative curvature.} J. Funct. Anal. 75 (1987), no. 2, 260--310.

\bibitem{Mettler2} T. Mettler. {\em On K\"{a}hler metrisability of two-dimensional complex projective structures.} Monatsh. Math. 174 (2014), no. 4, 599--616.

\bibitem{Mettler} T. Mettler. {\em Weyl metrisability of two-dimensional projective structures.} Math. Proc. Cambridge Philos. Soc. 156 (2014), no. 1, 99--113.

\bibitem{Mikes} J. Mike\v{s}. {\em Holomorphically projective mappings and their generalizations}. Geometry, 3. J. Math. Sci. (New York) 89 (1998), 1334--1353.

\bibitem{Mik} J. Mike\v{s}. {\em Geodesic mappings of affine-connected and Riemannian spaces}. Jour. Math. Sci. 78 (1996) 311--333.

\bibitem{OtsukiTashiro} T. Otsuki, Y. Tashiro. {\em On curves in K\"{a}hlerian spaces}. Math. J. Okayama Univ.
4 (1954), 57--78.

\bibitem{Pap} G. Papadopoulos. {\em Killing-Yano equations and G structures}. Classical Quantum Gravity 25 (2008), no. 10, 105016, 8 pp.

\bibitem{Penrose} R. Penrose, W. Rindler. {\em Spinors and Space-time}. Vol. 1, Cambridge University Press 1984.

\bibitem{Sharpe} R. W. Sharpe. {\em Cartan's Generalization of Klein's Erlangen Program}. Springer-Verlag, New York. 1997. x+426 pp. 

\bibitem{Sinj} N. S. Sinjukov. {\em Geodesic mappings of Riemannian spaces}. (Russian) [Geodesic mappings of Riemannian spaces] ``Nauka'', Moscow, 1979.

\bibitem{Tanaka} N. Tanaka. {\em A differential geometric study on strongly pseudo-convex manifolds}. Lectures in Mathematics, Department of Mathematics, Kyoto University, No.9. Kinokuniya Book-Store Co., Ltd., Tokyo, 1975.

\bibitem{Z1} S. Zelditch. {\em Local and global analysis of eigenfunctions on Riemannian manifolds}. Handbook of geometric analysis. No. 1, 545--658, Adv. Lect. Math. (ALM), 7, Int. Press, Somerville, MA, 2008.

\bibitem{Z2} S. Zelditch. {\em Eigenfunctions and nodal sets}. Surveys in differential geometry. Geometry and topology, 237-308, Surv. Differ. Geom., 18, Int. Press, Somerville, MA, 2013.


\end{thebibliography}
\end{document}